\documentclass[12pt,a4paper,reqno]{amsart}%
\usepackage{amssymb}
\usepackage{amsmath}
\usepackage{amsfonts}
\usepackage[T1]{fontenc}
\usepackage{color}
\usepackage[usenames,dvipsnames]{xcolor}
\usepackage{bbm}
\usepackage{graphicx}
\usepackage{rotating}
\usepackage{hyperref}
\usepackage{mathrsfs}
\usepackage[all]{xy}%
\providecommand{\U}[1]{\protect\rule{.1in}{.1in}}
\newtheorem{theorem}{Theorem}

\newtheorem{corollary}[theorem]{Corollary}

\newtheorem{definition}[theorem]{Definition}
\newtheorem{example}[theorem]{Example}

\newtheorem{lemma}[theorem]{Lemma}

\newtheorem{proposition}[theorem]{Proposition}
\newtheorem{remark}[theorem]{Remark}

\newcommand{\gM}{\mathcal{M}}
\newcommand{\gE}{\mathcal{E}}

\newcommand{\gL}{\mathcal{L}}

\newcommand{\oq}{\textquotedblleft}
\newcommand{\cq}{\textquotedblright}
\newcommand{\mr}{\mathrm}
\newcommand{\N}{\mathbb{N}}
\newcommand{\modu}{\, \mathrm{mod} \,}

\thanks{}
\email{lvitagliano@unisa.it}
\begin{document}
\title[Vector Valued Forms on $\N Q$-manifolds]{Vector Bundle Valued Differential Forms on $\N Q$-manifolds}
\author{Luca Vitagliano}
\address{DipMat, Universit\`a degli Studi di Salerno {\& Istituto Nazionale di Fisica
Nucleare, GC Salerno,} Via Giovanni Paolo II n${}^\circ$ 123, 84084 Fisciano (SA), Italy.}

\begin{abstract}
Geometric structures on $\N Q$-manifolds, i.e.~non-negatively graded manifolds with an homological vector field, encode non-graded geometric data on Lie algebroids and their higher analogues. A particularly relevant class of structures consists of vector bundle valued differential forms. Symplectic forms, contact structures and, more generally, distributions are in this class. We describe vector bundle valued differential forms on non-negatively graded manifolds in terms of non-graded geometric data. Moreover, we use this description to present, in a unified way, novel proofs of known results, and new results about degree one $\N Q$-manifolds equipped with certain geometric structures, namely symplectic structures, contact structures, involutive distributions (already present in literature) and locally conformal symplectic structures, and generic vector bundle valued higher order forms, in particular presymplectic and multisymplectic structures (not yet present in literature).
\end{abstract}
\maketitle


\ 


\tableofcontents

\section{Introduction}
Graded geometry encodes (non-graded) geometric structures in an efficient way. For instance, a vector bundle is the same as a degree one non-negatively graded manifold. In this respect $\N Q$-manifolds, i.e.~non-negatively graded manifolds equipped with an homological vector field, are of a special interest. Namely, they encode Lie algebroids in degree one, and higher versions of Lie algebroids (including homotopy Lie algebroids) in higher degrees \cite{v10} (see also \cite{b11, k01, sss09, sz11, bp12, vit14b}, and \cite{k01b, vit14} for applications of homotopy Lie algebroids). Accordingly, geometric data on Lie algebroids, or higher versions of them, that are compatible with the algebroid structure, can be encoded by suitable geometric structures on an $\N Q$-manifold that are preserved by the homological vector field. This is a general rule with various examples scattered in the literature. For instance, degree one \emph{symplectic $\N Q$-manifolds} are equivalent to Poisson manifolds (which can be understood as Lie algebroids of a special kind)  \cite{r02}. Similarly, degree one \emph{contact $\N Q$-manifolds} \cite{g13, m13} are equivalent to Jacobi manifolds, and degree one $\N Q$-manifolds equipped with a compatible involutive distribution are equivalent to Lie algebroids equipped with an \emph{IM foliation} (see \cite{jo11} for a definition) \cite{zz12}. More examples can be presented in degree two. For instance, degree two symplectic $\N Q$-manifolds are equivalent to Courant algebroids \cite{r02}, and degree two contact $\N Q$-manifolds encode a contact version of Courant algebroids: \emph{Grabowski's contact-Courant algebroids} \cite{g13}.

In all examples above the geometric structure on the $\N Q$-manifold is, or can be understood as, a \emph{differential form with values in a vector bundle}. This motivates the study of vector bundle valued differential forms (vector valued forms, in the following) on graded manifolds, and, in particular, $\N Q$-manifolds. In this paper, we describe vector valued forms on non-negatively graded manifolds in terms of non-graded geometric data (Theorem \ref{Theorem1}). Later we apply this description to the study of degree one $\N Q$-manifolds equipped with a compatible vector valued form. In this way, we get a unified formalism for the description of degree one contact $\N Q$-manifolds, symplectic $\N Q$-manifolds, and $\N Q$-manifolds equipped with a compatible involutive distribution. In particular, we manage to present alternative proofs of Roytenberg \cite{r02}, Grabowski \cite{g13}, Mehta \cite{m13}, and Zambon-Zhu \cite{zz12} results (in degree one). We also discuss three new examples. Namely, we show that degree one presymplectic $\N Q$-manifolds (with an additional non-degeneracy condition) are basically equivalent to Dirac manifolds (Corollary \ref{Cor:presymp}). We also show that degree one locally conformal symplectic $\N Q$-manifolds are equivalent to locally conformal Poisson manifolds (Theorem \ref{TheoremLCS}), and, more generally,
degree one $\N Q$-manifolds equipped with a compatible, higher degree, vector valued form are equivalent to Lie algebroids equipped with \emph{Spencer operators} (Theorem \ref{TheorSO}). The latter have been recently introduced in \cite{css12} (see also \cite{s13}) as infinitesimal counterparts of multiplicative vector valued forms on Lie groupoids. In particular, degree one multisymplectic $\N Q$-manifolds are equivalent to Lie algebroids equipped with an \emph{IM multisymplectic structure} \cite{bci13} (Theorem \ref{TheorMS}). 
We stress that we do only consider differential forms with values in vector bundles generated in one single degree (which, up to a shift, are actually generated in degree zero). This hypothesis simplifies the discussion a lot. We hope to discuss the general case, as well as higher degree cases, elsewhere.

The paper is divided into three main sections and two appendixes. In Section \ref{SecVVF}, after a short review of vector valued Cartan calculus on graded manifolds, we present our description of vector valued forms on $\N$-manifolds in terms of non-graded geometric data (Theorem \ref{Theorem1}). As already remarked, this description allows one to present in a unified way various results scattered in the literature about the correspondence between geometric structures on degree one $\N Q$-manifolds and (non-graded) geometric structures on Lie algebroids. In Section \ref{Sec1}, we discuss $1$-forms on degree one $\N Q$-manifolds. Surjective $1$-forms are the same as distributions and we discuss in some details the contact and involutive cases. The results of this section (Theorem \ref{Gra} and Theorem \ref{TheoremLCS}) are already present in literature, but they are presented here in a new and unified way that allows a straightforward generalization to (possibly degenerate) differential forms of higher order. In Section \ref{Sec2}, we discuss $2$-forms (on degree one $\N Q$-manifolds). In particular, we present a novel proof of Roytenberg's remark that degree one symplectic $\N Q$-manifolds are equivalent to Poisson manifolds \cite{r02} (Theorem \ref{Roy}). We also generalize Roytenberg result in two different directions, namely to presymplectic forms on one side (Theorem \ref{Theor:presymp} and Corollary \ref{Cor:presymp}) and to locally conformal symplectic structures on the other side (Theorem \ref{TheoremLCS}). In Section \ref{SecHigher} we discuss the general case of a differential form of arbitrarily high order. In particular, we relate compatible vector valued forms on $\N Q$-manifold and the Spencer operators of Crainic-Salazar-Struchiner \cite{css12,s13} (Theorem \ref{TheorSO}). Finally, we discuss degree one multisymplectic $\N Q$-manifolds (Theorem \ref{TheorMS}). The paper is complemented by two appendixes. In Appendix \ref{Ap1}, we revisit slightly the concept of locally conformal symplectic manifolds \cite{v85}, and give a slightly more intrinsic definition of them. We also briefly review the relation between locally conformal symplectic manifolds and locally conformal Poisson manifolds \cite{v07}. In Appendix \ref{ApLA}, we review the definition of Lie algebroids and their representations. As already remarked they play a key role in the paper. 

\subsection{Notations and Conventions}
Let $V = \bigoplus_i V_i$ be a graded vector space. We denote by $|v|$ the degree of a homogeneous element, i.e.~$|v| = i$ whenever $v \in V_i$, unless otherwise stated.

Let $\gM$ be a (graded) manifold, and $\gE \rightarrow \gM$ a (graded) vector bundle on it. We denote by $M$ the support of $\gM$. In the case when $\gM$ is non-negatively graded, $M$ is also the degree zero shadow of $\gM$. Moreover, we denote by $C^\infty_i (\gM)$, (resp.~$\mathfrak{X}_i (\gM)$, $\Gamma_i (\gE)$) the vector space of degree $i$ smooth functions on $\gM$ (resp.~vector fields on $\gM$, sections of $\gE$). We also denote by $\mathfrak{X}_- (\gM)$ the graded vector space of negatively graded vector fields on $\gM$. Sometimes, if there is no risk of confusion, we denote by $\gE$ the (graded) $C^\infty (\gM)$-module of sections of $\gE$. Similarly, we often identify (graded) vector bundle morphisms and (graded) homomorphisms between modules of sections. 

We adopt the Einstein summation convention.

\section{Vector Valued Forms on Graded Manifolds}\label{SecVVF}
\subsection{$\N Q$-manifolds and vector $\N Q$-bundles}
 we refer to \cite{r02,m06,cs11} for details about graded manifolds, and, in particular, $\N$-manifolds. In the following, we just recall some basic facts which will be often used below. We will work with the simplest possible notion of a \emph{graded manifold}. Namely, any graded manifold $\gM$ in this paper is equipped with one single $\mathbb{Z}$-grading in its algebra $C^\infty (\gM)$ of smooth functions (unless otherwise stated). Moreover, $C^\infty (\gM)$ is graded commutative with respect to the grading. We will call $\emph{degree}$ the grading. We will focus on $\N$-manifolds, i.e.~\emph{non-negatively graded} manifolds. Recall that the \emph{degree of an $\N$-manifold} is the highest degree of its coordinates. Similarly, the degree of a vector $\N$-bundle, i.e.~a non-negatively graded vector bundle over an $\N$-manifold, is the highest degree of its fiber coordinates.

\begin{example} \label{Ex3}
Every degree one $\N$-manifold $\gM$ is of the form $A[1]$ for some non-graded vector bundle $A \rightarrow M$, and one has $C^\infty (\gM) =\Gamma ( \wedge^\bullet A^\ast)$. In particular, degree zero functions on $\gM$ identify with functions on $M$, and degree one functions on $\gM$ identify with sections of $A^\ast$. Accordingly, degree $-1$ vector fields on $\gM$ identify with sections of $A$. In the following, we will tacitly understand the identifications $C^\infty_0 (\gM) \simeq C^\infty (M)$, $C_1^\infty (\gM) \simeq \Gamma( A^\ast)$, and $\mathfrak{X}_{-1}(\gM) \simeq \Gamma (A)$. The action of a degree $-1$ vector field $X \in \Gamma(A)$ on a degree one function $f \in \Gamma (A^\ast)$ is given by the \emph{duality pairing}: $X(f) = \langle X, f \rangle$.
\end{example}

\begin{example}\label{Ex2}
Recall that every $\N$-manifold $\gM$ is fibered over its degree zero shadow $M$. Every degree zero vector $\N$-bundle $\gE$ over $\gM$ is of the form $\gM \times_M E$ for some non-graded vector bundle $E \rightarrow M$, and one has $\Gamma (\gE) = C^\infty (\gM) \otimes \Gamma(E)$ (where the tensor product is over $C^\infty (M)$). In particular, degree zero sections of $\gE$ identify with sections of $E$. In the following, we will tacitly understand the identification $\Gamma_0 (\gE) \simeq \Gamma (E)$.
\end{example}

A \emph{$Q$-manifold} is a graded manifold $\gM$ equipped with an homological vector field $Q$, i.e.~a degree one vector field $Q$ such that $[Q,Q] = 0$. An $\N Q$-manifold is a non-negatively graded $Q$-manifold.

\begin{example}
Every degree one $\N Q$-manifold $(\gM, Q)$ is of the form $(A[1], d_A)$ for some non-graded Lie algebroid $A \rightarrow M$ (see Appendix \ref{ApLA} for a definition of Lie algebroid). Here $d_A$ is the homological derivation induced in $\Gamma (\wedge^\bullet A^\ast) = C^\infty (\gM)$. The Lie bracket $[\![-,-]\!]$ in $\Gamma (A)$ and the anchor $\rho : \Gamma(A) \rightarrow \mathfrak{X}(M)$ can be recovered from $Q$ via formulas
\begin{align*}
[\![X,Y]\!] & = [[Q,X],Y], \\
\rho (X) (f) & = [Q, X] (f),
\end{align*}
where $X,Y \in \Gamma (A)$ are also interpreted as degree $-1$ vector fields on $\gM$ (so that $[[Q,X],Y]$ is a degree $-1$ vector field as well), and $f \in C^\infty (M)$.
\end{example}

Similarly, we call a \emph{$Q$-vector bundle} (resp. $\N Q$-vector bundle) a graded vector bundle $\gE \rightarrow \gM$ (resp. a vector $\N$-bundle) equipped with an homological derivation. In this respect, recall that a (graded) derivation of $\gE$ is a graded, $\mathbb{R}$-linear map ${\mathbb X} : \Gamma (\gE) \rightarrow \Gamma (\gE)$ such that
\[
{\mathbb X} (fe) = X(f) e + (-)^{f {\mathbb X}} f {\mathbb X} (e), \quad f \in C^\infty (\gM), \quad e \in \Gamma (\gE),
\]
for a (necessarily unique) vector field $X \in \mathfrak{X} (\gM)$ called the \emph{symbol of} ${\mathbb X}$. Clearly, a derivation of $\gE$ is completely determined by its symbol and its action on generators of $\Gamma (\gE)$.
\begin{example}
Denote by $\Delta$ the grading vector field on $\gM$, i.e.~$\Delta (f) = |f| f $, for all homogeneous functions $f$ on $\gM$.
The grading $\Delta_{\gE} : \Gamma (\gE) \rightarrow \Gamma (\gE)$, $e \mapsto |e| e$, is a distinguished degree zero derivation. Obviously, the symbol of $\Delta_{\gE}$ is $\Delta$.
\end{example}

\begin{example} \label{Ex1}
Let $\gM$ be an $\N$-manifold, and let $\gE = \gM \times_M E$ be a degree zero vector $\N$-bundle on it. Since $\Gamma (\gE)$ is generated in degree zero, then a negatively graded derivation $\mathbb X$ of $\gE$ is completely determined by its symbol $X$ and, therefore, it is the same as a negatively graded vector field on $\gM$. Specifically, for a section of $\gE$ of the form $f \otimes e$, $f \in C^\infty (\gM)$, $e \in \Gamma (E)$, one has ${\mathbb X} (f \otimes e) = X(f) \otimes e$. In the following, we will tacitly identify negatively graded derivations of $\gE$ and negatively graded vector fields on $\gM$.
\end{example}

Derivations of $\Gamma (\gE)$ are sections of a (graded) Lie algebroid $D\gE$ over $\gM$ with bracket given by the (graded) commutator, and anchor given by the symbol. An \emph{homological derivation of $\gE$} is a degree one derivation $\mathbb{Q}$, with symbol $Q$, such that $[\mathbb{Q},\mathbb{Q}] = 0$ (in particular, $Q$ is an homological vector field).

\begin{example}
Any degree zero $\N Q$-vector bundle $(\gE, \mathbb{Q})$ over a degree one $\N$-manifold $\gM$ is of the form $(A[1]\times_M E, d_E)$ for some non-graded Lie algebroid $A \rightarrow M$ equipped with a representation $E \rightarrow M$. Here $d_E$ is the homological derivation induced on $\Gamma (\wedge^\bullet A^\ast \otimes E) = \Gamma (\gE)$. The algebroid structure on $A$ corresponds to the symbol $Q$ of $\mathbb{Q}$, while the (flat) $A$-connection $\nabla^E$ in $E$ can be recovered from $\mathbb{Q}$ via formula
\[
\nabla^E_X e  = [\mathbb{Q}, X] (e) = \mathbb{Q} (X(e))
\]
where $X \in \Gamma (A)$ is also interpreted as a degree $-1$ derivation of $\gE$ (see Example \ref{Ex1}), and $e \in \Gamma (E)$.
\end{example}

\subsection{Vector valued Cartan calculus on graded manifolds}\label{SecCartan}
Let $\gM$ be a graded manifold and let $\gE$ be a graded vector bundle over $\gM$. Differential forms on $\gM$ are functions on $T[1]\gM$ which are polynomial on fibers of $T[1] \gM \rightarrow \gM$. In particular, the algebra $\Omega (\gM)$ of differential forms on $\gM$ is equipped with two gradings: the \oq form\cq\ degree and the \oq internal, manifold\cq\ degree, which is usually referred simply as the \emph{degree} (or, sometimes, the \emph{weight}). The \oq total\cq\ degree is the sum of the form degree and the degree. Notice that the algebra $\Omega (\gM)$ is graded commutative with respect to the total degree. Similarly, $\gE$-valued differential forms on $\gM$ are sections of the vector bundle $T[1] \gM \times_{\gM} \gE \rightarrow T[1] \gM$ which are polynomial on fibers of $T[1] \gM \longrightarrow \gM$. The $\Omega (\gM)$-module $\Omega (\gM, \gE) \simeq \Omega (\gM) \otimes \Gamma (\gE)$ of $\gE$-valued forms is equipped with two gradings, the \oq internal\cq\ degree will be referred to simply as the \emph{degree}.  we will denote by $|\omega|$ the degree of a homogeneous (with respect to the internal degree) $\gE$-valued form $\omega$.

Now, we briefly review the $\gE$-valued version of Cartan calculus. Let ${\mathbb X}$ be a derivation of $\gE$. There are unique derivations $i_{\mathbb X} , L_{\mathbb X}$ of the vector bundle $T[1] \gM \times_{\gM} \gE \rightarrow T[1] \gM$ such that
\begin{enumerate}
\item the symbol of $i_{\mathbb X}$ is the insertion $i_{X}$ of the symbol $X$ of ${\mathbb X}$,
\item $i_{\mathbb X}$ vanishes on $\Gamma (\gE)$,
\item the symbol of $L_{\mathbb X}$ is the Lie derivative $L_{X}$ along the symbol $X$ of ${\mathbb X}$,
\item $L_{\mathbb X}$ agrees with ${\mathbb X}$ on $\Gamma (\gE)$.
\end{enumerate}
Notice that, actually,  $i_{\mathbb X}$ does only depend on the symbol of ${\mathbb X}$. For this reason, we will sometimes write $i_{X}$ for $i_{\mathbb X}$.
\begin{example}
For any homogenous $\gE$-valued form $\omega$, $L_{\Delta_{\gE}} \omega = |\omega| \omega$.
\end{example}
The following $\gE$-valued Cartan identities hold
\begin{equation} \label{Eq6}
[i_{\mathbb X}, i_{{\mathbb X}^\prime}] = 0, \quad
[L_{\mathbb X}, i_{{\mathbb X}^\prime}] = i_{[ {\mathbb X}, {\mathbb X}^\prime]}, \quad
[L_{\mathbb X}, L_{{\mathbb X}^\prime}] = L_{[ {\mathbb X}, {\mathbb X}^\prime]}.
\end{equation}
for all ${\mathbb X}, {\mathbb X}^\prime \in \Gamma (D \gE )$, where the bracket $[-,-]$ denotes the graded commutator. Moreover,
\begin{equation} \label{Eq5}
i_{f {\mathbb X}} = f i_{\mathbb X}, \quad L_{f {\mathbb X}} = f L_{\mathbb X} + (-)^{f+ {\mathbb X}} df\, i_{\mathbb X} .
\end{equation}
for all $f \in C^\infty (\gM)$.

Now suppose that $\gE$ is equipped with a flat connection $\nabla$. Recall that a \emph{connection} in $\gE$ is a graded, homogeneous, $C^\infty (\gM)$-linear map $\nabla : \mathfrak{X}(\gM) \rightarrow \Gamma (D \gE )$, denoted $X \mapsto \nabla_X$, such that the symbol of $\nabla_X$ is precisely $X$. In particular $| \nabla | = 0$. Derivation $\nabla_X$ is called the \emph{covariant derivative along $X$}. A connection $\nabla$ is \emph{flat} if it is a morphism of (graded) Lie algebras, i.e.~$[\nabla_X, \nabla_Y] = \nabla_{[X,Y]}$, for all $X,Y \in \mathfrak{X} (\gM)$. A connection $\nabla$ in $\gE$ determines a unique degree one derivation $d_\nabla$ of the vector bundle $T[1] \gM \times_\gM \gE \rightarrow T[1] \gM$ such that
\begin{enumerate}
\item the symbol of $d_\nabla$ is the de Rham differential $d \in \mathfrak{X} (T[1] \gM)$,
\item $i_X d_\nabla e = \nabla_X e$ for all $e \in \Gamma (\gE)$ and $X \in \mathfrak{X}(\gM)$.
\end{enumerate}
Derivation $d_\nabla$ is the \emph{de Rham differential of $\nabla$}. It is an homological derivation iff $\nabla$ is flat.

Let $\nabla$ be a flat connection in $\gE$. The following identities hold
\begin{equation}\label{eq:nabla}
[i_{\nabla_X}, d_\nabla]= L_{\nabla_X}, \quad [L_{\nabla_X} , d_\nabla] = 0, \quad [d_\nabla, d_\nabla] = 0,
\end{equation}
for all $X \in \mathfrak{X} (\gM)$.

\begin{remark}\label{rem:Delta_E}
Specialize to the case when $\gM$ is an $\N$-manifold and $\gE$ is degree zero. Then $\gE = \gM \times_M E$ for some vector bundle $E$ over the degree zero shadow $M$ of $\gM$. A connection $\nabla^0$ in $E$ induces a unique connection $\nabla$ in $\gE$ such that
\[
\nabla_X e = \nabla^0_{\underline{X}} e,
\]
for all $e \in \Gamma (E)$ and $X \in \mathfrak{X}_0 (\gM)$, where $\underline X \in \mathfrak X (M)$ is the projection of $X$ onto $M$. Connection $\nabla$ is flat iff $\nabla^0$ is flat. Moreover, every connection in $\gE$ is of this kind. Notice that, whatever $\nabla$, the covariant derivative along the grading vector field $\Delta$ coincides with the grading derivation $\Delta_\gE$. To see this it is enough to compare the action of $\nabla_\Delta$ and $\Delta_\gE$ on generators. Locally, $\Gamma (\gE)$ is generated by $\nabla^0$-flat sections of $E$. Thus, let $e \in \Gamma (E)$ be $\nabla^0$-flat. Then $\nabla_\Delta e = 0 = \Delta_\gE e$. As an immediate consequence, every $d_\nabla$-closed $E$-valued differential form on $\gM$ of positive degree $n$, $\omega$, is also $d_\nabla$-exact, i.e.~$\omega = d_\nabla \vartheta$, for a suitable $\vartheta$. One can choose, for instance, $\vartheta = n^{-1} i_{\Delta_\gE} \omega$. Indeed,
\[
d_\nabla \left( \frac{1}{n} i_{\Delta_\gE} \omega \right) = \frac{1}{n} [d_\nabla, i_{\Delta_\gE}] \omega = \frac{1}{n} L_{\Delta_\gE} \omega = \omega.
\]
\end{remark}

\subsection{An alternative description of vector valued forms on $\N $-manifolds} In the following, we will only consider the case when $\gM$ is an $\N $-manifold and $\gE$ is generated in one single degree. Let $M$ be the degree zero shadow of $\gM$. Then, up to an irrelevant shift, $\gE$ is isomorphic to a pull-back $\gM \times_M E$, where $E$ is a non-graded vector bundle over $M$ (see Example \ref{Ex2}). Accordingly, we will often write $C^\infty (\gM, E)$ for $\Gamma (\gE)$ and $\Omega (\gM, E)$ for $\Omega (\gM, \gE)$.

\begin{remark}
Despite the huge simplifications inherent to the hypothesis $\gE \simeq \gM \times_M E$, this case still captures many interesting situations. For instance, degree $n$ symplectic \cite{r02} and contact \cite{g13,m13} $\N$-manifolds can be both understood as $\N $-manifolds $\gM$ equipped with a degree $n$ differential form with values in a vector bundle concentrated in just one degree (the trivial bundle $\gM \times \mathbb{R}$ in the symplectic case, and a generically non-trivial line bundle concentrated in degree $n$ in the contact case). We hope to discuss the case of a general vector bundle $\gE$ elsewhere.
\end{remark}

\begin{theorem}\label{Theorem1}
Let $n$ be a positive integer. A degree $n$ differential $k$-form on $\gM$ with values in $E$ is equivalent to the following data:
\begin{itemize}
\item a degree $n$ (first order) differential operator $D : \mathfrak{X}_- (\gM) \rightarrow \Omega^k (\gM, E)$, and
\item a degree $n$ $C^\infty (M)$-linear map $\ell : \mathfrak{X}_- (\gM) \rightarrow \Omega^{k-1} (\gM, E)$,
\end{itemize}
such that
\begin{equation}
D(f X) = f D(X) +(-)^X df\, \ell (X), \label{Eq1}
\end{equation}
and, moreover,
\begin{align}
L_X D(Y) - (-)^{XY} L_Y D(X) & = D([X,Y]),  \label{Eq2}\\
L_X \ell(Y) - (-)^{X(Y-1)} i_Y D(X) & = \ell([X,Y]), \label{Eq3}\\
i_X \ell(Y) - (-)^{(X-1)(Y-1)} i_Y \ell(X) & = 0. \label{Eq4}
\end{align}
for all $X,Y \in \mathfrak{X}_- (\gM)$, and $f \in C^\infty (M)$.
\end{theorem}

\begin{remark}
By induction on $n$, Theorem \ref{Theorem1} provides a description of $E$-valued differential forms in terms of non-graded data. Indeed, $D$ and $\ell$ take values in lower degree forms and one can use degree zero forms, namely $E$-valued forms on $M$ as base of induction.
\end{remark}

\begin{proof}
Let $\omega$ be a degree $n$, $E$-valued differential $k$-form on $\gM$. Define $D : \mathfrak{X}_- (\gM) \rightarrow \Omega^k (\gM, E)$ and $\ell : \mathfrak{X}_- (\gM) \rightarrow \Omega^{k-1} (\gM, E)$ by putting
\begin{equation}\label{Eq7}
D(X) := L_X \omega, \quad \text{and} \quad \ell(X):= i_X \omega, \quad X \in \mathfrak{X}_- (\gM).
\end{equation}
Properties (\ref{Eq1}), (\ref{Eq2}), (\ref{Eq3}), and (\ref{Eq4}) immediately follow from identities (\ref{Eq6}), and (\ref{Eq5}).

Conversely, let $D$ and $\ell$ be as in the statement of the theorem, and prove that there exists a unique degree $n$ differential form $\omega \in \Omega^k (\gM, E)$ fulfilling (\ref{Eq7}). We propose a local proof. One can pass to the global setting by partition of unity arguments. Let $\ldots, z^a, \ldots$ be positively graded coordinates on $\gM$ and $\partial_a := \partial / \partial z^a$. In particular the grading derivation $\Delta_\gE$ is locally given by
\[
\Delta_\gE = |z^a| z^a \partial_a .
\]
Moreover, $\mathfrak{X}_-(\gM)$ is locally generated, as a $C^\infty (M)$-module, by vector fields
\[
z^{b_1} \cdots z^{b_k} \partial_b, \quad |z^{b_1}| + \cdots + |z^{b_k}| - |z^b|  < 0.
\]
Put
\[
\omega := \frac{|z^a|}{n} \left(  z^a D(\partial_a) + dz^a \ell(\partial_a) \right)
\]
and prove (\ref{Eq7}). First of all, for $X = X^a \partial_{a}$,
\begin{align*}
L_X \omega & = \frac{|z^a|}{n}L_X \left( z^a D(\partial_a) + dz^a \ell(\partial_a) \right) \\
                    & = \frac{|z^a|}{n} \left( X^a D(\partial_a) + (-)^{z^a X} L_X D(\partial_a) \right. \\
                    & \quad\  \left. + (-)^X dX^a \ell (\partial_a) + (-)^{X(z^a+1)}dz^a L_X \ell(\partial_a) \right) .
\end{align*}
In view of  (\ref{Eq2}) and (\ref{Eq3}),
\[
(-)^{z^aX}L_X D(\partial_a) = L_{\partial_a} D(X) - D([\partial_a,X])
\]
and
\[
\quad (-)^{X(z^a+1)} L_X \ell(\partial_a) = i_{\partial_a}D(X) - (-)^{X}\ell([\partial_a,X]),
\]
so that
\begin{align}
L_X \omega  \nonumber & = \frac{1}{n} L_{\Delta_{\gE}} D(X) + \frac{|z^a|}{n} \left( X^a D(\partial_a) + (-)^X dX^a \ell(\partial_a) \right. \nonumber \\
                                       & \quad\ \left. - z^a D([\partial_a,X]) - (-)^Xdz^a \ell([\partial_a,X])    \right)  \nonumber \\
                     & = \frac{n+|X|}{n}  D(X) + \frac{|z^a|}{n} \left( X^a D(\partial_a) + (-)^X dX^a \ell(\partial_a) \right. \nonumber \\
                      & \quad\ \left. - z^a D([\partial_a,X]) - (-)^Xdz^a\ell([\partial_a,X])  \right). \label{Eq8}
\end{align}
Similarly
\begin{equation}
i_X \omega = \frac{n+|X|}{n} \ell (X) + \frac{|z^a|}{n} \left( X^a \ell (\partial_a) - z^a \ell ([\partial_a,X]) \right). \label{Eq9}
\end{equation}
Using $X = \partial_b$ in Equations (\ref{Eq8}) and (\ref{Eq9}) gives
\begin{equation}
L_{\partial_b} \omega = D(\partial_b), \quad \text{and} \quad i_{\partial_b} \omega = \ell (\partial_b). \label{Eq10}
\end{equation}
In view of identity (\ref{Eq1}), in order to prove (\ref{Eq7}), it is enough to restrict to vector fields $X$ of the form $z^{b_1} \cdots z^{b_k} \partial_b$. This case can be treated by induction on $k$, using Equations (\ref{Eq10}) as base of induction. Namely, use $X =  z^{b_1} \cdots z^{b_k} \partial_b$ in (\ref{Eq8}), with $k>0$. Since
\[
[\partial_a,  X] = \sum_i (-)^{(b_1 + \cdots + b_i) z^a} \delta^{b_i}_a z^{b_1} \cdots \widehat{z^{b_i}} \cdots z^{b_k} \partial_b,
\]
where a hat \oq $\widehat{-}$\cq\ denotes omission, by induction hypothesis we have
\[
D([\partial_a, X]) =L_{[\partial_a, X]} \omega, \quad \text{and} \quad \ell([\partial_a, X])=i_{[\partial_a, X]}\omega .
\]
A direct computation shows that the second summand in the right hand side of (\ref{Eq8}) is equal to
$
- \frac{|X|}{n} L_X \omega
$.
Similarly, the second summand in the right hand side of (\ref{Eq9}) is equal to
$
- \frac{|X|}{n} i_X \omega
$.
Notice that, since $k>0$, then $|X| > -n$ and one can conclude that $L_X \omega = D(X)$, and, similarly, $i_X \omega = \ell(X)$.

To prove uniqueness, it is enough to show that a degree $n$ differential form $\omega$ with values in $E$ is completely determined by contraction with and Lie derivative along negatively graded derivations. Thus,
\[
n \omega = L_{\Delta_\gE} \omega = |z^a| ( z^a L_{\partial_a} \omega + dz^a i_{\partial_a} \omega).
\]
In particular $\omega$ is completely determined by $L_{\partial_a} \omega$ and $i_{\partial_a} \omega$.
\end{proof}

We will refer to the data $(D, \ell)$ corresponding to a vector valued form $\omega$ as the \emph{Spencer data} of $\omega$. Indeed, as we will show in Section \ref{SecHigher}, they are a vast generalization of the \emph{Spencer operators} considered in \cite{css12,s13}.

\begin{example}\label{Examp1bis}
Let $E \to M$ be a non-graded vector bundle equipped with a flat connection $\nabla$, and let $\gM$ be an $\N$-manifold. As discussed in Section \ref{SecCartan}, $\nabla$ induces a flat connection in the graded vector bundle $\gM \times_M E \to \gM$ which we denote again by $\nabla$. In its turn, the induced connection determines an homological derivation $d_\nabla$ of the vector bundle $T[1]\gM \times_M E \rightarrow T[1]\gM$ of $E$-valued forms on $\gM$. Notice that $d_\nabla$ maps $k$-forms to $(k+1)$-forms. Now, let $\omega \in \Omega^k (\gM, E)$ and let $(D, \ell)$ be the corresponding Spencer data. We want to describe the Spencer data $(D', \ell')$ of $d_\nabla \omega$. To do this, we first observe that a discussion similar to that in Remark \ref{rem:Delta_E} shows that, whatever $\nabla$, the covariant derivative along a negatively graded vector field $X \in \mathfrak X_- (\gM)$ satisfies $\nabla_X = X$. Hence, from (\ref{eq:nabla})
\[
D' (X) = L_X d_\nabla \omega = L_{\nabla_X} d_\nabla \omega = d_\nabla L_X \omega = d_\nabla D(X)
\]
and
\[
\ell' (X) = i_X d_\nabla \omega = i_{\nabla_X} d_\nabla \omega = L_X \omega - (-)^{|X|} d_\nabla i_X \omega = D(X) - (-)^{|X|} d_\nabla \ell (X),
\]
which completely describe $(D',\ell')$ in terms of $(D,\ell)$ and $d_\nabla$.
\end{example}

\begin{example}\label{Examp2}
Let $E \rightarrow M$ be a non-graded vector bundle. The first jet bundle $J^1 E \rightarrow M$ fits in an exact sequence 
\begin{equation}\label{Eq11}
0 \longrightarrow \Omega^1 (M,E) \longrightarrow \Gamma (J^1 E) \overset{p}{\longrightarrow} \Gamma (E) \longrightarrow 0
\end{equation}
of $C^\infty (M)$-linear maps, where $p$ is the canonical projection. Sequence (\ref{Eq11}) splits (beware, \emph{over $\mathbb R$} not over $C^\infty (M)$) via the universal first order differential operator $j^1 : \Gamma (E) \rightarrow \Gamma (J^1 E)$. Accordingly, there is a first order differential operator $S : \Gamma(J^1 E) \rightarrow \Omega^1 (M,E)$ sometimes called the \emph{Spencer operator}.
The degree $n$ $\N$-manifold $\gM = J^1 E[n] $ comes equipped with an $E$-valued, degree $n$ \emph{Cartan $1$-form} $\theta$. In order to define $\theta$ recall that negatively graded vector fields on $\gM$ are concentrated in degree $-n$, and $\mathfrak{X}_{-n} (\gM)$ identifies with $\Gamma (J^1 E)$ as a $C^\infty (M)$-module. Now, $\theta$ is uniquely defined by the following properties
\begin{equation}
i_{j^1 e} \theta = e, \quad \text{and} \quad L_{j^1 e} \theta = 0, \label{Eq12}
\end{equation}
for all $e \in \Gamma (E)$. It immediately follows from (\ref{Eq12}) that the Spencer data $(D, \ell)$ of $\theta$ identify with $(-)^n$ times the Spencer operator $S : \Gamma (J^1 E) \rightarrow \Omega^1 (M,E)$ and the projection $\Gamma(J^1 E) \rightarrow \Gamma (E)$ respectively.
\end{example}

\begin{example}\label{Examp1}
Let $M$ be a non-graded manifold. The degree $n$ $\N$-manifold $\gM = T^\ast [n] M$ comes equipped with the obvious tautological, degree $n$ $1$-form $\vartheta$. Consider the degree $n$ $2$-form $\omega = d\vartheta$. Negatively graded vector fields on $\gM$ are concentrated in degree $-n$, and $\mathfrak{X}_{-n} (\gM)$ is naturally isomorphic to $\Omega^1 (M)$ as a $C^\infty (M)$-module. It is easy to see that $\omega$ is uniquely defined by the following properties
\begin{equation}\label{Eq13}
i_{df} \omega = df, \quad \text{and} \quad L_{df} \omega = 0,
\end{equation}
for all $f \in C^\infty (M)$. It immediately follows from (\ref{Eq13}) that the Spencer data $(D,\ell)$ of $\omega$ identify with $(-)^n$ times the exterior differential $d : \Omega^1 (M) \rightarrow \Omega^2 (M)$ and the identity $\mr{id} : \Omega^1 (M) \rightarrow \Omega^1 (M)$ respectively.
\end{example}

\begin{example}\label{Examp3}
Let $E \rightarrow M$ be a non-graded vector bundle equipped with a flat connection $\nabla$. The degree $n$ $\N$-manifold $\gM = T^\ast [n] M \otimes E$ is equipped with a tautological, degree $n$ $E$-valued $1$-form $\vartheta$. Flat connection $\nabla$ induces a flat connection in the graded vector bundle $\gM \times_M E \rightarrow \gM$ which we denote again by $\nabla$. Consider the homological derivation $d_\nabla$ as in Example \ref{Examp1bis}. Notice that $d_\nabla$ agrees with the de Rham differential of $\nabla$ on degree zero forms, i.e.~elements of $\Omega (M,E)$. Consider the degree $n$ $2$-form $\omega = d_\nabla \vartheta$ with values in $E$. Negatively graded vector fields on $\gM$ are concentrated in degree $-n$, and $\mathfrak{X}_{-n} (\gM)$ is isomorphic to $\Omega^1 (M,E)$ as a $C^\infty (M)$-module. It is easy to see that $\omega$ is uniquely defined by the following properties
\begin{equation}\label{Eq14}
i_{d_\nabla e} \omega = d_\nabla e, \quad \text{and} \quad L_{d_\nabla e} \omega = 0,
\end{equation}
for all $e \in \Gamma (E)$. It immediately follows from (\ref{Eq14}) that the Spencer data $(D,\ell)$ of $\omega$ identify with $(-)^n$ times the de Rham differential $d_\nabla : \Omega^1 (M,E) \rightarrow \Omega^2 (M,E)$ and the identity $\mr{id} : \Omega^1 (M,E) \rightarrow \Omega^1 (M,E)$ respectively.
\end{example}

In the three remaining sections we use Theorem \ref{Theorem1} (and Proposition \ref{prop:comp} below) to describe degree one $\N Q$-manifolds equipped with a compatible vector valued differential form (see below) in terms of non-graded data. In particular, we manage to give alternative proofs of known results about compatible contact structures \cite{g13,m13}, involutive distributions \cite{zz12}, and symplectic forms \cite{r02} on degree one $\N Q$-manifolds. We also manage to find new results about compatible, presymplectic and locally conformal symplectic structures, and, more generally, higher order vector valued forms on degree one $\N Q$-manifolds. It turns out (Theorem \ref{TheorSO}) that a compatible degree one differential $k$-form on a degree one $\N Q$-manifold $(\gM , Q)$ is equivalent to a Lie algebroid equipped with a structure recently identified in \cite{css12} as the infinitesimal counterpart of a multiplicative vector valued form on a Lie groupoid (see also \cite{s13}), namely, a \emph{$k$-th order Spencer operator}.

Let $\gM$ be an $\N$-manifold, with degree zero shadow $M$, and let $(\gE, \mathbb{Q})$ be an $\N Q$-vector bundle over it. We denote by $Q$ the symbol of $\mathbb Q$.

\begin{definition}\label{def:comp}
An $\gE$-valued differential form on $\gM$, $\omega$, is \emph{compatible with $\mathbb{Q}$} if $L_\mathbb{Q} \omega = 0$. 
\end{definition}

Suppose $\gE$ is degree zero. Then $\gE = \gM \times_M E$ for a non-graded vector bundle $E \to M$. For later use, we conclude this section expressing the compatibility of an $E$-valued form $\omega$ on $\gM$ with $\mathbb{Q}$ in terms of Spencer data.

\begin{proposition}\label{prop:comp}
Let $\omega \in \Omega^k (\gM, E)$ be a degree $n > 0$, $E$-valued $k$-form on $\gM$, and let $(D, \ell)$ be its Spencer data. Then $\omega$ is compatible with $\mathbb Q$, i.e.~$L_{\mathbb Q}\omega = 0$, iff
\begin{align}
A(X,Y):= {}& D ([[Q, X], Y]) - L_{[\mathbb Q, X]} D(Y) \nonumber \\
               & - (-)^{|X||Y|} \left(L_{[\mathbb Q, Y]} D(X) - L_{\mathbb Q} L_Y D(X) \right)  = 0, \label{C1}
\end{align}
\begin{align}
 B(X,Y) :={} & \ell ([[Q, X], Y]) - (-)^{|Y|}i_{[Q, X]} D(Y) \nonumber \\
 & - (-)^{|X||Y|} \left(L_{[\mathbb Q, Y]} \ell(X)  - L_{\mathbb Q} L_Y \ell(X) \right) = 0, \label{C2}
\end{align}
\begin{align}
 C(X,Y)  := i_{[Q, X]} \ell (Y) + (-)^{(|X|-1)(|Y| -1)} \left(i_{[Q, Y]} \ell (X) - L_{\mathbb Q} i_Y \ell (X) \right)  = 0. \label{C3}
\end{align}
\end{proposition}

\begin{remark}
By induction on $n$, Proposition \ref{prop:comp} provides a description of the compatibility condition between $\omega$ and $\mathbb Q$ in terms of non-graded data (see also Remark \ref{rem:homotopy} below). Notice that, when $|\omega| = 1$, the last summand in (\ref{C1}), (\ref{C2}) and (\ref{C3}) vanishes by degree reasons.
\end{remark}

\begin{proof}
First of all, notice that $[[\mathbb Q, X], Y]$ is negatively graded for all $X,Y$. Hence it identifies with $[[Q,X], Y]$. In particular, the left hand side of (\ref{C1}), (\ref{C2}) and (\ref{C3}) are well-defined. Now, for any $\omega$ as in the statement, $L_{\mathbb Q} \omega$ is a degree $n+1$ form. Since every positive degree form on $\gM$ is completely determined by its Spencer data, then $L_{\mathbb Q} \omega$ vanishes iff 
\[
L_YL_X L_{\mathbb Q} \omega = i_Y L_X L_{\mathbb Q} \omega = L_Y i_X L_{\mathbb Q} \omega = i_Y i_X L_{\mathbb Q} \omega = 0,
\]
for all $X,Y \in \mathfrak X_- (\gM)$. It immediately follows from the second Cartan identity (\ref{Eq6}) that condition $i_Y L_X L_{\mathbb Q} \omega = 0$ is actually redundant. It remains to compute $L_YL_X L_{\mathbb Q} \omega$, $L_Y i_X L_{\mathbb Q} \omega$, and $i_Y i_X L_{\mathbb Q} \omega$. So
\begin{align*}
 L_Y L_X L_\mathbb{Q} \omega  = {} &L_Y L_{[X,\mathbb{Q}]} \omega -(-)^{|Y|} L_Y L_\mathbb{Q} L_X \omega \\
                                                    = {} &(-)^{|X|+|Y|(|X|+1)}\left( L_{[[\mathbb{Q},X],Y]} \omega  -  L_{[\mathbb{Q},X]} L_Y \omega \right) \\
                                                    & - (-)^{|X|+|Y|} \left( L_{[\mathbb{Q},Y]} L_X \omega - L_{\mathbb Q} L_Y L_X \omega \right),
\end{align*}
which differs from $A(X,Y)$ in  (\ref{C1}) for an overall sign $(-)^{|X|+|Y|(|X|+1)}$.
Similarly,
\begin{align*}
 L_Y i_X L_\mathbb{Q} \omega  = {} & (-)^{|X|}L_Y i_{[\mathbb{Q}, X]} \omega - (-)^{|X|}L_Y L_\mathbb{Q} i_X \omega \\
                                                    ={} &  (-)^{(|X|+1)(|Y|+1)}\left(i_{[[\mathbb{Q},X],Y]} \omega - (-)^{|Y|} i_{[\mathbb{Q},X]} L_Y \omega \right) \\
                                                    & + (-)^{|X|+|Y|} \left(  L_{[\mathbb{Q},Y]} i_X \omega - L_{\mathbb Q} L_Y i_X \omega \right),
\end{align*}
which differs from $B(X,Y)$ in (\ref{C2}) for an overall sign $(-)^{(|X|+1)(|Y|+1)}$. Finally,
\begin{align*}
 i_Y i_X L_\mathbb{Q} \omega
											& = (-)^{|X|} i_Y i_{[\mathbb Q,X]} \omega -(-)^{|X|} i_Y L_\mathbb{Q} i_X \omega \\
                                                   	& = (-)^{|X| |Y|} i_{[\mathbb{Q},X]}  i_Y \omega - (-)^{|X|+|Y|}\left( i_{[\mathbb{Q},Y]} i_X - L_{\mathbb Q}i_Yi_X \omega \right),                                          
\end{align*}
which differs from $C(X,Y)$ in (\ref{C3}) for an overal sign $(-)^{|X||Y|}$. This concludes the proof.
\end{proof}

\begin{remark}\label{rem:homotopy}
When $\gM$ is the total space of a negatively graded vector bundle $V \to M$ (which is always the case up to a non-canonical isomorphism), an homological vector field on $\gM$ is the same as an $L_\infty$-algebroid structure on $\Gamma (V^\ast)$ (see, e.g., \cite{bp12, b11, sss09, vit14b}). We conjecture the existence of formulas expressing the compatibility between $\omega$ and $\mathbb Q$ in terms of the higher brackets (and the anchor) of this $L_\infty$-algebroid, and the Spencer data of $\omega$. Similarly, when no isomorphism $\gM \simeq V$ is assigned, there should be formulas involving \emph{Getzler higher derived brackets} on $\mathfrak X_- (\gM)$ \cite{G10}. Finding these formulas goes beyond the scopes of this paper and we postpone this task to a subsequent publication.
\end{remark}

\section{Vector Valued $1$-forms on $\N Q$-manifolds}\label{Sec1}
\subsection{Vector valued $1$-forms and distributions}
Let $\gM$ be a degree $n$ $\N$-manifold, $n>0$, and let $(\mathcal{E} = \gM \times_M E , \mathbb{Q})$ be a degree zero $\N Q$-vector bundle over it. According to Definition \ref{def:comp}, a degree $n$ $1$-form $\theta$ with values in $E$ is \emph{compatible} with $\mathbb{Q}$ if, by definition, $L_\mathbb{Q} \theta = 0$. Several interesting geometric structures are described by compatible $1$-forms. For instance, compatible distributions on an $\N Q$-manifold are equivalent to surjective compatible $1$-forms. Namely, Let $(\gM, Q)$ be a degree $n$ $\N Q$-manifold, and let $\mathcal{D} \subset T\gM$ be a distribution on $\gM$. Consider the normal bundle $T\gM / \mathcal{D}$. Projection $T \gM \rightarrow T \gM / \mathcal{D}$ can be interpreted as a degree zero, surjective $1$-form with values in $T \gM / \mathcal{D}$. We say that $\mathcal{D}$ is \emph{co-generated} in degree $k$ if $T \gM / \mathcal{D}$ is generated in degree $-k$. In this case, $T \gM / \mathcal{D} = \gM \times_M E[k]$ for a suitable non-graded vector bundle $E \rightarrow M$, and the projection $\theta_\mathcal{D} : T \gM \rightarrow \gM \times_M E$ can be interpreted as a degree $k$, surjective, $E$-valued $1$-form such that $\ker \theta_\mathcal{D} = \mathcal{D}$. Conversely, if $E \rightarrow M$ is a non-graded vector bundle and $\theta$ is a degree $k$, surjective $1$-form with values in $E$, then $\mathcal{D} := \ker \theta$ is a distribution such that $\theta_\mathcal{D} = \theta$.
\begin{definition}
A distribution $\mathcal{D}$ on $\gM$, co-generated in degree $n$, is \emph{compatible} with $Q$ if $[Q, \Gamma (\mathcal{D})] \subset \Gamma (\mathcal{D})$.
\end{definition}
Now, let $\mathcal{D}$ be a distribution co-generated in degree $n$ and $E$ be such that $T \gM / \mathcal{D} = \gM \times_M E [-n]$. If $\mathcal{D}$ is compatible with $Q$, then the commutator with $Q$ restricts to an homological derivation of $\mathcal{D}$, hence it also descends to an homological derivation of the vector bundle $\gM \times_M E $ which we denote by $\mathbb{Q}$. Moreover, $L_\mathbb{Q}\theta_\mathcal{D} = 0$.  Indeed, for every vector field $X \in \mathfrak X (\gM)$,
\begin{equation}\label{eq:L_Qtheta}
i_X L_{\mathbb Q} \theta_\mathcal{D} = (-)^{|X|} \left(i_{[Q,X]} \theta_\mathcal{D} - \mathbb Q (i_X \theta_\mathcal{D}) \right) = 0.
\end{equation}
 
Conversely, if $(\gE , \mathbb{Q})$ is a degree zero $\N Q$-bundle and $\theta$ is a degree $n$ surjective $1$-form with values in $E$, then , it follows from (\ref{eq:L_Qtheta}) that  $L_\mathbb{Q} \theta = 0$ iff 1) $\ker \theta$ is a distribution compatible with the symbol $Q$ of $\mathbb{Q}$, and 2) $\mathbb{Q}$ is induced on $\Gamma (\gE)$ by the adjoint operator $[Q , -]$ on $\mathfrak{X}(\gM)$. One concludes that compatible distributions are the same as compatible surjective $1$-forms.

\begin{remark}\label{rem:symmetries}
The above discussion is actually independent of the degree of $Q$. Hence, it shows that an infinitesimal symmetry of $\mathcal{D}$, i.e.~any vector field $X$ such that $[X, \Gamma (\mathcal{D})] \subset \Gamma (\mathcal{D})$, determines a derivation ${\mathbb X}$ of $T\gM/\mathcal{D}$ via 
\[
{\mathbb X} ( Y \modu \mathcal{D} ) := [X,Y] \modu \mathcal{D} .
\]
The symbol of ${\mathbb X}$ is precisely $X$. Moreover, one can compute the Lie-derivative $L_{\mathbb X} \theta_{\mathcal D}$, and find $L_{\mathbb X} \theta_{\mathcal D} = 0$.
\end{remark}

Now recall that a distribution $\mathcal{D}$ on a (graded) manifold $\gM$ comes equipped with a \emph{curvature form}
\[
\omega_\mathcal{D} : \Gamma (\mathcal{D}) \times \Gamma (\mathcal{D}) \longrightarrow \Gamma (T\gM / \mathcal{D}), \quad (X,Y) \longmapsto [X,Y]\ \mr{mod}\ \Gamma (\mathcal{D}).
\]
The curvature form $\omega_\mathcal{D}$ measures how far is $\mathcal{D}$ from being involutive. The two limit cases $\omega_\mathcal{D}$ non-degenerate and $\omega_\mathcal{D} = 0$ are of a special interest. The first one corresponds to \emph{maximally non-integrable distributions}, the second one to \emph{involutive distributions}.
 
\subsection{Degree one contact $\N Q$-manifolds}
Let $\gM$ be an $\N$-manifold and let $C$ be an hyperplane distribution on it. Since $\gL := T\gM / C$ is a line bundle, then it is generated in one single degree. A \emph{degree $n$ contact $\N$-manifold} is an $\N$-manifold $\gM$ equipped with a \emph{degree $n$ contact structure}, i.e.~an hyperplane distribution $C$, such that the line bundle $\gL := T\gM / C$ is generated in degree $n$, and the curvature form $\omega_C$ is non-degenerate (see \cite{g13} for an alternative definition exploiting the ``symplectization trick'').

\begin{example}
Let $L\rightarrow M$ be a non-graded line bundle. The kernel of the Cartan form $\theta$ on $J^1 L[n]$ (see Example \ref{Examp2}) is a degree $n$ contact structure.
\end{example}

It follows from the definition that, if $(\gM, C)$ is a degree $n$ contact $\N$-manifold, then the degree of $\gM$ is at most $n$. When $\gL$ is a trivial line bundle, $C$ is the kernel of a (no-where vanishing) one-form $\alpha$ which can be used to simplify the theory significantly (see \cite{m13}). In this case the contact structure is said to be \emph{co-orientable} and a choice of $\alpha$ provides a co-orientation (i.e.~an orientation of $\gL$). In the general case, $\gL := \gM \times_M L[n]$ for a non-graded line bundle $L \rightarrow M$, and $C$ is the kernel of a (degree $n$) $1$-form $\theta_C$ with values in a (generically non trivial) line bundle $L$.

A degree $n$ contact structure on $\gM$ determines a non-degenerate degree $-n$ Jacobi bracket $\{-,-\}$ on $\Gamma (\gL)$, i.e.~a degree $-n$ Lie bracket which is a graded first order differential operator in each entry and such that the associated morphism $J^1 \gL \otimes J^1 \gL \rightarrow \gL$ is non-degenerate (see also Appendix \ref{Ap1}). For the details about how to define the Jacobi bracket $\{-,-\}$ from $C$ in the non-graded case see, for instance, \cite{cs13}. The generalization to the graded case can be carried out straightforwardly and the obvious details are left to the reader. A \emph{degree $n$ contact $\N Q$-manifold} is a degree $n$ contact manifold $(\gM , C)$ equipped with an homological vector field $Q$ such that $[Q, \Gamma (C)] \subset \Gamma (C) $, in other words it is an $\N Q$-manifold equipped with a compatible degree $n$ contact structure. If $(\gM, C, Q)$ is a contact $\N Q$-manifold, the homological vector field $Q$ induces an homological derivation $\mathbb{Q}$ of $\gL$ as discussed above. Thus,
equivalently, a degree $n$ contact $\N Q$-manifold is a degree $n$ contact manifold $(\gM , C)$ equipped with an homological derivation $\mathbb{Q}$ of $\gL$ such that $L_{\mathbb{Q}} \theta_C = 0$.

\begin{theorem}[Mehta \cite{m13} (in the co-oriented case only) and, independently, Grabowski \cite{g13} (in the general case)] \label{Gra}
Every degree one contact $\N$-manifold $(\gM, C)$ is of the kind $(J^1 L[1], \ker \theta)$, up to contactomorphisms, where $L \rightarrow M$ is a (non-graded) line bundle, and $\theta$ is the Cartan form on $J^1 L [1]$. Moreover, there is a one-to-one correspondence between degree one contact $\N Q$-manifolds and (non-graded) manifolds equipped with an abstract Jacobi structure (see the appendixes).
\end{theorem}

Notice that Mehta does only discuss the case when $C$ is co-orientable, i.e.~$T\gM/C$ is globally trivial. Moreover, he selects a contact form, which amount to selecting a global trivialization $T\gM/C \simeq \gM \times \mathbb{R}[1]$ (see \cite{m13} for details). On the other hand, Grabowski discusses the general case (he actually treats the degree two case as well). His proof relies on the ``symplectization trick'' which consists in understanding a contact manifold as a homogeneous symplectic manifold (see \cite{g13}) and then using already known results in the symplectic case. We propose an alternative proof avoiding the ``symplectization trick'' and focusing on the Spencer data of the structure $1$-form of $C$. We refer to \cite{cs13} for details on abstract Jacobi structures.

\begin{proof}
Let $(\gM = A[1], C)$ be a degree one contact $\N$-manifold, and let $\gL = T\gM / C$ be the associated degree one line-bundle. Then $\gL = \gM \times_M L[1]$ for a non-graded line bundle $L \rightarrow M$, and $\theta_C$ is a degree one $L$-valued $1$-form on $\gM$. Denote by $(D, \ell)$ the Spencer data of $\theta_C$. The Jacobi bracket $\{-,-\}$ determines a degree $-1$ graded vector bundle isomorphism between $J^1\gL$ and $D \gL$. Since negatively graded derivations are completely determined by their symbol, this gives an isomorphism $\Gamma (J^1 L) \simeq \mathfrak{X}_{-1}(\gM) \simeq \Gamma (A)$, hence a diffeomorphism $\gM \simeq J^1 L[1]$. It is easy to see that diagram 
\[
\xymatrix { \Gamma(L) \ar@{=}[d] & \Gamma (A) \ar[r]^-D \ar[l]_-{\ell} \ar@{-}[d] \ar@{}@<0.2ex>[d]^{\begin{sideways}$\widetilde{\quad\quad}$\end{sideways}} 
& \Omega^1 (M,L) \ar@{=}[d] \\
                  \Gamma (L) & \Gamma (J^1 L)  \ar[r]^-{-S} \ar[l]_-{p} &  \Omega^1 (M,L) 
}
\]
commutes. This shows that diffeomorphism $\gM \simeq J^1 L [1]$ identifies $\theta_C$ with the Cartan form $\theta$ (see Example \ref{Examp2}), thus proving the first part of the statement. In the following we identify $\gM$ and $J^1 L[1]$. For the second part of the statement, let $\mathbb{Q}$ be an homological derivation of $\gL$ and let $Q$ be its symbol. Moreover, let $(J^1 L, \rho, [\![-,-]\!])$ and $(L, \nabla^L)$ be the Lie algebroid and the Lie algebroid representation associated to $\mathbb{Q}$. We use Proposition \ref{prop:comp} to see when is $(\mathcal{M}, C, \mathbb{Q})$ a contact $\N Q$-manifold. Since $\theta$ is a one form, (\ref{C3}) is automatically satisfied, and $\theta$ is compatible with $\mathbb Q$ iff $A(X,Y) = B(X,Y) = 0$, with $\omega = \theta$ and $X,Y \in \mathfrak{X}_{-1} (\gM) \simeq \Gamma (J^1 L)$. In fact, one can even restrict to $X,Y$ in the form $j^1 \lambda, j^1 \mu$, with $\lambda, \mu \in \Gamma (L)$. In this case, one gets
\[
A(X,Y) = D([[Q,j^1 \lambda], j^1 \mu]) = -S [\![ j^1 \lambda , j^1 \mu ]\!] , 
\]
and
\begin{align*}
B(X,Y) = \ell ([[Q,j^1 \lambda], j^1 \mu]) + L_{[\mathbb Q, j^1 \mu]} \lambda & = p [\![ j^1 \lambda, j^1 \mu]\!] + \nabla^L_{j^1 \mu} \lambda,
\end{align*}
where we used that $\ell (j^1 \lambda) = i_{j^1 \lambda} \theta = \lambda$, and $D (j^1 \lambda) = L_{j^1 \lambda} \theta = 0$    (see Example \ref{Examp2}). Concluding, $(\gM, C, \mathbb{Q})$ is a contact $\N Q$-manifold iff $ p [\![ j^1 \lambda, j^1 \mu]\!] = - \nabla^L_{j^1 \mu} \lambda = \nabla^L_{j^1 \lambda} \mu $ and $S [\![ j^1 \lambda , j^1 \mu ]\!] = 0$, i.e.~iff $(J^1 L, [\![ -,-]\!], \rho)$ is the Lie algebroid associated to a Jacobi structure on $L \rightarrow M$, and $\nabla^L$ is its natural representation (see Appendix \ref{ApLA}).
\end{proof}

\subsection{Involutive distributions on degree one $\N Q$-manifolds}
Compatible involutive distributions (co-generated in degree one) on a degree one $\N Q$-manifold are equivalent to infinitesimally multiplicative (IM) foliations of a special kind. Let $(A, [\![-,-]\!], \rho)$ be a Lie algebroid over a manifold $M$, and let $F \subset TM$ be an involutive distribution. An \emph{IM foliation of $A$ over $F$} \cite{jo11} is a triple consisting of 
\begin{itemize}
\item involutive distribution $F$,
\item a Lie subalgebroid $B \subset A$,
\item a flat $F$-connection $\nabla$ in the quotient bundle $A/B$,
\end{itemize}
such that 
\begin{enumerate}
\item sections $X$ of $A$ such that $X\modu B$ is $\nabla$-flat form a Lie subalgebra in $\Gamma (A)$ with sections of $B$ as a Lie ideal, \label{IMf2}
\item $\rho$ takes values in the stabilizer of $F$,
\item $\rho |_B$ takes values in $F$.
\end{enumerate}
As the terminology suggests, IM foliations are infinitesimal counterparts of involutive multiplicative distributions on Lie groupoids \cite{jo11}. Zambon and Zhu \cite{zz12} proved that IM foliations can be also understood as degree one $\N Q$-manifolds equipped with an involutive distribution preserved by the homological vector field. In the following, we restrict to distributions co-generated in degree one. In this particularly simple situation, we can provide an alternative proof of Zambon-Zhu result exploiting the description of vector valued forms in terms of their Spencer data.

\begin{lemma} \label{IMf6}
Let $(A, [\![-,-]\!], \rho)$ be a Lie algebroid over a manifold $M$. If $(TM, B, \nabla)$ is an IM foliation of $A$ over $TM$, then there is a flat $A$-connection $\nabla^{A/B}$ in $A/B$ such that
\begin{equation}\label{IMf1}
\nabla_X^{A/B} (Y \modu B) = \nabla_{\rho (Y)} (X \modu B) - [\![ Y,X]\!] \modu B,
\end{equation}
and, moreover,
\begin{equation} \label{IMf3}
d_\nabla \left( [\![X,Y]\!] \modu B \right) =  L_{\nabla_X^{A/B}} d_\nabla \left( Y \modu B \right) - L_{\nabla_Y^{A/B}} d_\nabla \left( X \modu B \right).
\end{equation}
for all $X,Y \in \Gamma (A)$.
  Conversely, if $B \subset A$ is a vector subbundle, $\nabla$ is a flat connection in $A/B$, and $\nabla^{A/B}$ is a flat $A$-connection in $A/B$ satisfying (\ref{IMf1}) and (\ref{IMf3}), then $(TM, B, \nabla)$ is an IM foliation of $A$ over $TM$.  
\end{lemma}
\begin{proof}
For the first part of the statement, let $(TM,B, \nabla)$ be an IM foliation as in the statement. Denote by $\Gamma_\nabla$ the sheaf on $M$ consisting of sections $X$ of $A$ such that $X \modu B$ is $\nabla$-flat. Since $\Gamma (A/B)$ is locally generated by flat sections, $\Gamma (A)$ is locally generated by $\Gamma_\nabla$. Now, the left hand side of (\ref{IMf1}) is clearly $C^\infty (M)$-linear in $X$. Moreover, it vanishes whenever $Y \in \Gamma (B)$. To see this, it is enough to compute on local generators $X \in \Gamma_\nabla$. In this case, the left hand side of (\ref{IMf1}) reduces to $-[\![ Y, X ]\!] \modu B$ which vanishes by property (\ref{IMf2}) of IM foliations whenever $Y \in \Gamma (B)$. One concludes that (\ref{IMf1}) defines a differential operator $\nabla_X^{A/B}$ in $\Gamma (A/B)$ for all $X \in \Gamma (A)$. It is easy to see that, besides being $C^\infty (M)$-linear in $X$, $\nabla_X^{A/B}$ is actually a derivation with symbol $\rho(X)$. Thus $\nabla^{A/B}$ is a well-defined $A$-connection in $A/B$. To see that it is flat, check that the curvature
\[
R(X,Y) (Z \modu B) := \left( [\nabla_X^{A/B}, \nabla_Y^{A/B}] - \nabla^{A/B}_{[\![ X, Y ]\!]} \right) (Z \modu B)
\]
vanishes on all $X,Y, Z$. Since $R$ is linear in the first two arguments, it is enough to check that it vanishes on $X,Y \in \Gamma_\nabla$. In this case $[\![ X,Y]\!] \in \Gamma_\nabla$ as well and
\[
R(X,Y) (Z \modu B) = [\![ [\![ Z,Y]\!] , X ]\!] -  [\![ [\![ Z,X]\!] , Y ]\!] + [\![ Z, [\![X,Y]\!] ]\!] = 0
\]
by Jacobi identity.

Finally, notice that Equation (\ref{IMf3}) is equivalent to
\begin{equation}\label{IMf4}
\begin{aligned}
& \nabla_Z \left( [\![ X,Y ]\!] \modu B \right) \\
& = \left( \nabla_X^{A/B} \nabla_Z - \nabla_{[\rho(X), Z]}\right) (Y \modu B) - \left( \nabla_Y^{A/B} \nabla_Z - \nabla_{[\rho(Y), Z]}\right) (X \modu B),
\end{aligned}
\end{equation}
$X,Y\in \Gamma (A)$, and $Z \in \mathfrak{X}(M)$. Actually, (\ref{IMf4}) can be easily obtained from (\ref{IMf3}), by inserting $Z$ in both sides, and using 
\[
[i_Z, L_{\nabla_X^{A/B}}] = i_{[Z, \rho (X)]}.
\]
Thus it is enough to check that the expression
\begin{align*}
& S(X,Y; Z) :=
\nabla_Z \left( [\![ X,Y ]\!] \modu B \right) \\
& - \left( \nabla_X^{A/B} \nabla_Z - \nabla_{[\rho(X), Z]}\right) (Y \modu B) + \left( \nabla_Y^{A/B} \nabla_Z - \nabla_{[\rho(Y), Z]}\right) (X \modu B)
\end{align*}
vanishes for all $X,Y,Z$. A direct computation shows that $S$ is $C^\infty (M)$-linear in its first two arguments. Therefore, it is enough to compute $S(X,Y; Z)$ for $X,Y \in \Gamma_\nabla$. In this case $[\![ X,Y]\!] \in \Gamma_\nabla$ as well and $S(X,Y;Z)$ vanishes.

The second part of the statement immediately follows from (\ref{IMf1}) and (\ref{IMf3}).  
\end{proof}

\begin{theorem}[Zambon \& Zhu \cite{zz12}]\label{TheoremLCS}
There is a one-to-one correspondence between degree one $\N Q$-manifolds equipped with a compatible involutive distribution, co-generated in degree one, and Lie algebroids $A \rightarrow M$ equipped with an IM foliation over $TM$.
\end{theorem}
\begin{proof}
Let $\gM = A[1]$ be a degree one $\N$-manifold, and let $\mathcal{D}$ be an involutive distribution on it, co-generated in degree one. Denote by $\pi: \gM \rightarrow M$ the projection of $\gM$ onto its zero dimensional shadow. The quotient bundle $T\gM / \mathcal{D} $ identifies with $\gM \times_M E[1]$ for a non-graded vector bundle $E \rightarrow M$, and $\theta_{\mathcal{D}}$ is a degree one $E$-valued $1$-form on $\gM$. Moreover, $\mathcal{D}$ projects surjectively onto $TM$, i.e.~$\pi_\ast \mathcal{D} = TM$. In particular, for any vector field $Z$ on $M$ there is a (degree zero) vector field $\tilde{Z} \in \Gamma (\mathcal{D})$ that is $\pi$-related to $Z$. 

Denote by $(D, \ell)$ the Spencer data of $\theta_{\mathcal{D}}$. In particular, $\ell : \Gamma (A) \rightarrow \Gamma (E)$ is surjective. Let $B = \ker \ell$ so that $E$ identifies with $A/B$ and $\ell$ identifies with the projection $\Gamma (A) \rightarrow \Gamma (A/B)$. In the following we will understand this isomorphism. There is a unique first order differential operator $\delta : \Gamma(A/B) \rightarrow \Omega^1 (M, A/B)$ such that diagram
\[
\xymatrix{\Gamma (A) \ar[r]^-D \ar[d] & \Omega^1 (M , A/B) \\
                 \Gamma (A/B) \ar[ur]_-{\delta} &
} 
\]
commutes. To see this it is enough to show that $\Gamma (B) \subset \ker D$. Since $\Gamma (B) = \Gamma_{-1} (\mathcal{D})$, and sections of $\mathcal{D}$ are infinitesimal symmetries by involutivity, then $D(X) = L_X \theta_{\mathcal{D}} = 0$ for all $X \in \Gamma (B)$ (see Remark \ref{rem:symmetries}). It follows from (\ref{Eq1}) that $\delta (f \alpha) =  f \delta \alpha - df \otimes \alpha$, for all $f \in C^\infty (M)$, and $\alpha \in \Gamma (A/B)$. Therefore, $\delta$ is minus the (first) de Rham differential of a unique connection $\nabla$ in $A/B$. We claim that $\nabla$ is a flat connection. Indeed, first of all, notice that for all $Z \in \mathfrak{X} (M)$ and $X \in \Gamma (A)$,
\[
\nabla_Z (X \modu B) = i_Z d_\nabla (X \modu B) = - i_Z D(X) = - i_{\tilde{Z}} L_X \theta_\mathcal{D}, 
\]
where $\tilde{Z}$ is any degree zero vector field on $\gM$ that is $\pi$-related to $Z$. We can choose $\tilde{Z} \in \Gamma(\mathcal{D})$ so that
\[
\nabla_Z (X \modu B) = -i_{\tilde{Z}} L_X \theta_\mathcal{D} = -i_{[X,\tilde{Z}]} \theta_\mathcal{D} = [\tilde{Z}, X] \modu B.
\]
Now, let $Y,Z$ be vector fields on $M$, and let $\tilde{Y}, \tilde{Z}$ be vector fields in $\mathcal{D}$ that are $\pi$-related to them. Then, by involutivity, $[\tilde{Y}, \tilde{Z}]$ is in $\mathcal{D}$ and it is $\pi$-related to $[Y,Z]$. Thus
\[
\begin{aligned}
\nabla_{[Y,Z]} (X \modu B) & = [[\tilde{Y},\tilde{Z}], X] \modu B \\
                                          & = \left( [\tilde{Y}, [\tilde{Z}, X]] - [\tilde{Z}, [\tilde{Y}, X]] \right) \modu B \\
                                          & = [\nabla_Y, \nabla_Z] (X \modu B).
\end{aligned}
\]
Conversely, let $B \subset A$ be a vector subbundle and let $\nabla$ be a flat connection in $A/B$. Denote by $\ell : A \rightarrow A/B$ the projection. Then $(-d_\nabla \circ \ell, \ell)$ are Spencer data for an $A/B$-valued $1$-form $\theta$ on $\gM $. Put $\mathcal{D} = \ker \theta$. To see that $\mathcal{D}$ is involutive, notice that $\Gamma_{-1} (\mathcal{D}) = \ker \ell = \Gamma (B)$. Moreover, $\mathcal{D}$ projects surjectively on $TM$, therefore $\Gamma (\mathcal{D})$ is generated by 1) sections of $B$ and 2) degree zero vector fields in $\mathcal{D}$ that are projectable onto $M$. Commuting the latter with the former, one gets sections of $B$ which are again in $\mathcal{D}$. It remains to show that the commutator of two projectable vector fields $\tilde{Z}, \tilde{Y}$ in $\mathcal{D}$ is again in $\mathcal{D}$, i.e.~$i_{[\tilde{Y}, \tilde{Z}]} \theta = 0$. Now $i_{[\tilde{Y}, \tilde{Z}]} \theta = 0$ iff $L_X i_{[\tilde{Y}, \tilde{Z}]} \theta = 0$ for all $X \in \mathfrak{X}_{-}(\gM) = \Gamma (A)$. The same computation as above shows that
\[
L_X i_{[\tilde{Y}, \tilde{Z}]} \theta = \left( [\nabla_Y, \nabla_Z] - \nabla_{[Y,Z]} \right) (X \modu B) = 0.
\]
where $Y= \pi_\ast \tilde{Y}$, and $Z = \pi_\ast \tilde{Z}$. We conclude that involutive distributions $\mathcal D$ on $\gM$ co-generated in degree one are equivalent to the following (non-graded) data: a vector subbundles $B \subset A$ and a flat connection in $A/B$. 

Finally, let $\mathcal D$ be an involutive distribution on $\gM$ co-generated in degree one and let $(B, \nabla)$ be the corresponding non-graded data. Moreover, let $\mathbb{Q}$ be an homological derivation of $T\gM/ \mathcal{D} = \gM \times_M A/B$, let $Q$ be its symbol, and let $(A, [\![-,-]\!], \rho)$ and $(A/B , \nabla^{A/B})$ be the Lie algebroid and the Lie algebroid representation corresponding to $\mathbb{Q}$. Distribution $\mathcal{D}$ is compatible with $Q$, and $\mathbb{Q}$ is induced by $[Q,-]$, iff   $\theta_\mathcal{D}$ is compatible with $\mathbb Q$. To see when is this the case, we use again Proposition \ref{prop:comp}. Identity (\ref{C3}) is automatically satisfied by $\omega = \theta_{\mathcal D}$. Additionally, for $\omega = \theta_{\mathcal D}$ and $X,Y \in \mathfrak{X}_{-} (\gM) = \Gamma (A)$, one gets
\begin{align*}
A(X,Y) = {} & D ([[Q, X], Y]) - L_{[\mathbb Q, X]} D (Y) + L_{[\mathbb Q, Y]} D (X)\\
 ={} & - d_\nabla \left( [\![X,Y]\!] \modu B \right) +  L_{\nabla_X^{A/B}} d_\nabla \left( Y \modu B \right)  - L_{\nabla_Y^{A/B}} d_\nabla \left( X \modu B \right) ,
\end{align*}
and
\begin{align*}
B(X,Y) = {} & \ell([[Q, X], Y])  + i_{[Q, X]} D(Y) + L_{[\mathbb Q, Y]} \ell (X) \\
 ={} & [\![ X,Y]\!] \modu B - \nabla_{\rho(X)} (Y \modu B) + \nabla_Y^{A/B} (X \modu B),
\end{align*}
  where we used that $D = -d_\nabla \circ \ell$. Proposition \ref{prop:comp} and Lemma \ref{IMf6} then show that $\mathcal{D}$ is compatible with $Q$, and $\mathbb{Q}$ is induced by $[Q, -]$, iff $(B, \nabla, TM)$ is an IM foliation of $A$ over $TM$, and $\nabla^{A/B}$ is the $A$-connection in the statement of the lemma. This concludes the proof.
\end{proof}

\section{Vector Valued $2$-forms on $\N Q$-manifolds}\label{Sec2}
\subsection{Degree one symplectic $\N Q$-manifolds}

Recall that a degree $n$ symplectic $\N$-manifold is an $\N$-manifold $\gM$ equipped with a degree $n$ symplectic structure, i.e.~a degree $n$ non-degenerate, closed $2$-form $\omega$. 

It immediately follows from the definition that, if $(\gM, \omega)$ is a degree $n$ symplectic $\N$-manifold, then the degree of $\gM$ is at most $n$. If $n > 0$, then $\omega = d \vartheta$, with $\vartheta = n^{-1} i_\Delta \omega$.

\begin{example}
The degree $n$ $2$-form $\omega$ on $T^\ast[n] M$ (see Example \ref{Examp1}) is a degree $n$ symplectic structure.
\end{example}

A degree $n$ symplectic $\N Q$-manifold is a degree $n$ symplectic manifold $(\gM, \omega)$ equipped with an homological vector field $Q$ such that $L_Q \omega = 0$.

\begin{theorem}[Roytenberg \cite{r02}]\label{Roy}
Every degree one symplectic $\N$-manifold $(\gM, \omega)$ is of the kind $(T^\ast [1] M, d\vartheta)$, up symplectomorpisms, where $\vartheta$ is the tautological degree one $1$-form on $T^\ast[1] M$ (see Example \ref{Examp1}). Moreover, there is a one-to-one correspondence between degree one symplectic $\N Q$-manifolds and (non-graded) Poisson manifolds.
\end{theorem}
Roytenberg proof \cite{r02} exploits explicitly the Poisson bracket. We propose an alternative proof focusing on Spencer data. The advantage is that we can apply the same strategy to degenerate (Theorem \ref{Theor:presymp} and Corollary \ref{Cor:presymp}) or higher degree forms (Theorems \ref{TheorSO} and \ref{TheorMS}) in a straightforward way. 
\begin{proof}
Let $(\gM, \omega)$ be a degree one symplectic $\N$-manifold, and let $(D, \ell)$ be the Spencer data of $\omega$. In particular, $\gM = A[1]$ for some vector bundle $A \rightarrow M$. By non-degeneracy $\ell : \mathfrak{X}_{-1} (\gM) \rightarrow \Omega^1 (\gM)$ is an isomorphism $\Gamma (A) \simeq \Omega^1 (M)$, i.e.~$\gM = A[1] \simeq T^\ast [1] M$. Moreover, since $\omega$ is closed, diagram
\begin{equation}\label{eq:diag_closed}
\xymatrix{\Gamma (A) \ar[r]^-D \ar[d]_-{\ell} & \Omega^2 (M) \\
                 \Omega^1 (M) \ar[ur]_-{-d} &
} 
\end{equation}
commutes. This shows that diffeomorphism $\gM \simeq T^\ast[1] M$ identifies $\omega$ with the canonical symplectic structure on $T^\ast [1] M$ (see Example \ref{Examp1}), thus proving the first part of the statement. In the following we identify $\gM$ and $T^\ast[1] M$. For the second part of the statement, let $Q$ be an homological vector field on $\gM$ and let $(T^\ast M, \rho, [\![-,-]\!])$ be the corresponding Lie algebroid. Similarly as in previous section, $(\gM, \omega, Q)$ is a symplectic $\N Q$-manifold iff   it satisfies (\ref{C2}), and (\ref{C3}), 
for all $X,Y \in \mathfrak{X}_{-1} (\gM) \simeq \Omega^1 (M)$. Indeed, since $\omega$ is closed, then condition (\ref{C1}) is actually a consequence of (\ref{C2}), and (\ref{C3}). It is easy to see that one can even restrict to $X,Y$ in the form $df,dg$, with $f,g \in C^\infty (M)$. In this case, one gets
 
\begin{align*}
B(X,Y) = \ell ([[Q,df], dg]) + L_{[Q,df]} \ell (dg)= [\![df,dg]\!] +L_{\rho (df)} dg = [\![df,dg]\!] +d \rho (df) (g),
\end{align*}
and
\begin{align*}
C(X,Y) = i_{[Q,df]} \ell (dg) + i_{[Q,dg]} \ell (df) = \rho (df) (g) + \rho (dg) (f),
\end{align*}
where we used that $\ell (df) = i_{df} \omega = df$, and $D(df) = L_{df} \omega = 0$   (see Example \ref{Examp1}). Concluding, $(\gM, \omega, Q)$ is a symplectic $\N Q$-manifold iff $\rho (df) (g) + \rho(dg) (f) = 0$ and $[\![dg,df]\!] = - d \rho (df) (g) = d \rho (dg) (f)$, i.e.~iff $(T^\ast M, \rho, [\![-,-]\!])$ is the Lie algebroid associated to a Poisson structure on $M$ (see Appendix \ref{ApLA}).
\end{proof}

\subsection{Degree one presymplectic $\N Q$-manifolds}

In this subsection we relax the hypothesis about non-degeneracy of the two form in the previous subsection. 

\begin{definition}
A \emph{degree $n$ presymplectic $\N$-manifold} is a degree $n$ $\N$-manifold $\gM$ equipped with a degree $n$ presymplectic structure, i.e.~a degree $n$ (possibly degenerate) closed $2$-form $\omega$. A \emph{degree $n$ presymplectic $\N Q$-manifold} is a degree $n$ presymplectic manifold $(\gM, \omega)$ equipped with an homological vector field $Q$ such that $L_Q \omega = 0$.
\end{definition}

\begin{remark}
Unlike the symplectic case, the existence of a presymplectic form on an $\N$-manifold $\gM$ doesn't bound the degree of $\gM$ anyhow. This is the reason why we added a condition on the degree of $\gM$ in the definition of a degree $n$ presymplectic $\N$-manifold above.
\end{remark}

In what follows we show that degree one presymplectic $\N Q$-manifolds (with an additional non-degeneracy condition) are basically equivalent to \emph{Dirac manifolds}. Recall that a Dirac manifold is a manifold $M$ equipped with a \emph{Dirac structure}, i.e.~a subbundle $\mathscr D \subset TM \oplus T^\ast M$ such that 1) $\mathscr D$ is maximally isotropic with respect to the canonical, split signature, symmetric form on $TM \oplus T^\ast M$
\begin{equation}\label{eq:pair}
\langle\!\langle (X, \sigma), (X', \sigma') \rangle\!\rangle = i_X \sigma' + i_{X'} \sigma,
\end{equation}
and 2) sections of $\mathscr D$ are preserved by the Dorfman (equivalently, Courant) bracket
\begin{equation}\label{eq:Dorf}
[ (X, \sigma), (X', \sigma') ]_{\mathsf{D}} := ([X,X'], L_X \sigma' - i_{X'} d \sigma),
\end{equation}
$X,X' \in \mathfrak X (M)$, $\sigma, \sigma' \in \Omega^1 (M)$. Any Dirac structure $\mathscr D$ is a Lie algebroid, with anchor given by projection $TM \oplus T^\ast M \to TM$ and bracket given by the Dorfman bracket (\ref{eq:Dorf}). Dirac manifolds encompass presymplectic and Poisson manifolds (see \cite{c90, b13} for more details). They are sometimes regarded as Lagrangian submanifolds in certain degree two symplectic $\N Q$-manifolds. Corollary \ref{Cor:presymp} below shows that they can be also regarded as suitable degree one presymplectic $\N Q$-manifolds.

Let $M$ be a manifold, denote by $\operatorname{pr}_T : TM \oplus T^\ast M \to TM$, and $\operatorname{pr}_{T^\ast} : TM \oplus T^\ast M \to T^\ast M$ the canonical projections.

\begin{theorem}\label{Theor:presymp}
There is a one-to-one correspondence between degree one presymplectic $\N Q$-manifolds and (non-graded) Lie algebroids $A \to M$ equipped with a vector bundle morphism $\Phi : A \to TM \oplus T^\ast M$ such that
\begin{enumerate}
\item the anchor of $A$ equals the composition $\operatorname{pr}_T \circ \Phi$,
\item the image of $\Phi$ is an isotropic subbundle with respect to (\ref{eq:pair}), and
\item $\Phi$ intertwines the Lie bracket $[\![-,-]\!]$ on $\Gamma (A)$ and bracket (\ref{eq:Dorf}) on $\Gamma (TM \oplus T^\ast M)$, i.e. $\Phi [\![X,Y]\!] = [ \Phi (X), \Phi (Y) ]_{\mathsf D}$ for all $X,Y \in \Gamma (A)$.
\end{enumerate}
\end{theorem}

\begin{proof}
Let $(\gM, \omega)$ be a degree one presymplectic $\N$-manifold, and let $(D, \ell)$ be the corresponding Spencer data. Moreover, let $Q$ be an homological vector field on $\gM$, and let $(A, \rho, [\![-,-]\!])$ be the corresponding Lie algebroid. Since $\omega$ is closed, diagram (\ref{eq:diag_closed}) commutes and $\omega$ is completely determined by $\ell$. Now, combine $\ell$ and $\rho : A \to TM$ in a vector bundle morphism $\Phi := (\rho , \ell) : A \to TM \oplus T^\ast M$. In particular, $\rho = \operatorname{pr}_T \circ \Phi$, i.e.~$\Phi$ satisfies property (1) in the statement. 
Similarly as in the proof of Theorem \ref{Roy}, $(\gM, \omega, Q)$ is a presymplectic $\N Q$-manifold iff   (\ref{C2}) and (\ref{C3}) are satisfied 
for all $X,Y \in \mathfrak{X}_{-1} (\gM) \simeq \Omega^1 (M)$. One gets
\begin{align*}
B(X,Y) & = \ell ([[Q, X], Y]) + i_{[Q,X]}D(Y) + L_{[Q,Y]} \ell (X) \\& = \ell([\![X,Y]\!]) - i_{\rho (X)} d \ell(Y) + L_{\rho(X)} \ell(X) \\
&=  \operatorname{pr}_{T^\ast} \left(\Phi [\![X,Y]\!] - [\Phi (X), \Phi (Y) ]_{\mathsf D} \right),
\end{align*}
and
\begin{align*}
C(X,Y) = i_{[Q,X]}\ell(Y) + i_{[Q,Y]}\ell (X) =  i_{\rho (X)} \ell (Y) - i_{\rho (Y)} \ell (X)  =  \langle\!\langle \Phi (X), \Phi (Y) \rangle \! \rangle .
\end{align*}
  Since $ \operatorname{pr}_{T} \left( [\Phi (X), \Phi (Y) ]_{\mathsf D} \right) = [\rho (X), \rho (Y)] = \rho [X, Y] =  \operatorname{pr}_{T} \Phi [\![X,Y]\!]$, one concludes that $(\gM, \omega, Q)$ is a presymplectic $\N Q$-manifold iff $\Phi$, besides satisfying property (1) in the statement, does also satisfy properties (2), (3).

Conversely, Let $\Phi : A \to TM \oplus T^\ast M$ be a vector bundle morphism. Put $\ell := \operatorname{pr}_{T^\ast} \circ \Phi$. It is easy to see that $(\ell,  - d \circ  \ell)$ is a pair of Spencer data corresponding to a degree one presymplectic form $\omega$ on $\gM$. If, additionally, $\Phi$ satisfies properties (1), (2), and (3) in the statement, then the same computations as above show that $(\gM, \omega, Q)$ is a presymplectic $\N Q$-manifold.
\end{proof}

\begin{corollary}\label{Cor:presymp}
There is a one-to-one correspondence between degree one presymplectic $\N Q$-manifolds $(\gM, \omega, Q)$ such that 
\begin{enumerate}
\item $\operatorname{rank} A = \dim M$, and
\item $\ker \ell \cap \ker \rho = 0$,
\end{enumerate}
(where $(A \to M, \rho, [\![-,-]\!])$ is the Lie algebroid corresponding to $(\gM, Q)$, and $(\ell, D)$ are Spencer data corresponding to $\omega$) and (non-graded) Lie algebroids $A \to M$ equipped with a Lie algebroid isomorphism $\Phi: A \simeq \mathscr D$ with values in a Dirac structure $\mathscr D \subset TM \oplus T^\ast M$ over $M$.
\end{corollary}

\begin{proof}
Let $(\gM, \omega, Q)$ be a degree one presymplectic $\N Q$-manifold and let $(A, \Phi)$ be the corresponding non-graded data as in Theorem \ref{Theor:presymp}. Vector bundle morphism $\Phi$ is injective iff condition (2) in the statement is satisfied. In this case, $\Phi$ is an isomorphism onto its image $\mathscr D$. Additionally, $\mathscr D$ is maximally isotropic in $TM \oplus T^\ast M$, hence a Dirac structure, iff $\operatorname{rank} \mathscr D = \operatorname{rank} A$ is precisely $\dim M$, i.e.~condition (1) in the statement is satisfied.
\end{proof}

\subsection{Degree one locally conformal symplectic $\N Q$-manifolds}

The original definition of a locally conformal symplectic (lcs) structure is (equivalent to) the following \cite{v85}: a \emph{lcs structure} on a manifold $M$ is a pair $(\phi, \omega)$, where $\phi$ is a closed $1$-form and $\omega$ is a non-degenerate $2$-form on $M$ such that $d\omega = \phi \wedge \omega$. Lcs manifolds share some properties with ordinary symplectic manifolds but are manifestly more general. Moreover, they are examples of Jacobi manifolds. In this paper we adopt an approach to lcs manifolds which is slightly more intrinsic than the traditional one (see Appendix \ref{Ap1}) in the same spirit as the \emph{intrinsic approach} to contact and Jacobi geometry of \cite{cs13}.

\begin{definition}
A degree $n$ \emph{abstract lcs $\N$-manifold} is an $\N$-manifold $\gM$ equipped with a degree $n$ \emph{abstract lcs structure}, i.e.~a triple $(\gL, \nabla, \omega)$ where $\gL \rightarrow \gM$ is a line $\N$-bundle, $\nabla$ is a flat connection in $\gL$, and $\omega$ is degree $n$, non-degenerate, $d_\nabla$-closed, $\gL$-valued $2$-form $\omega$. 
\end{definition}

First of all, notice that $\gL$, being a line bundle, is actually generated in one single degree $-k$. Up to a shift in degree in the above definition we may (and we actually will) assume $k = 0$. In particular, $\gL = \gM \times_M L$, for some (non-graded) vector bundle $L$ on the degree zero shadow $M$ of $\gM$, and $\nabla$ is actually induced from a flat connection on $L$. Exactly as in the symplectic case \cite{r02} one shows that, if $\gM$ possesses a degree $n$ abstract lcs structure, then, by non-degeneracy, the degree of $\gM$ is at most $n$. If $n > 0$, then $\omega = d_\nabla \vartheta$, with $\vartheta = n^{-1} i_{\Delta_\gL} \omega$.

\begin{example}
Consider the degree $n$ $2$-form $\omega$ of Example \ref{Examp3}. If $E = L$ is a line bundle then, triple $(T[1] \gM \times_M L, \nabla, \omega)$ is a degree $n$ abstract lcs structure.
\end{example}

 A degree $n$ abstract lcs symplectic $\N Q$-manifold is a degree $n$ abstract lcs manifold $(\gM, \gL, \nabla, \omega)$ equipped with an homological derivation $\mathbb{Q}$ of $\gL$ such that $L_\mathbb{Q} \omega = 0$. The proposition below shows that, actually, $\mathbb{Q}$ is completely determined by its symbol.
 
 \begin{proposition}
 Let $(\gM, \gL, \nabla, \omega)$ be an abstract lcs $\N$-manifold with homological derivation $\mathbb{Q}$, and let $Q$ be the symbol of $\mathbb{Q}$. Then $\mathbb{Q}$ is the covariant derivative along $Q$.
 \end{proposition}
 \begin{proof}
 Derivations $\mathbb{Q}$ and $\nabla_Q$ share the same symbol $Q$ and, therefore, their difference $\mathbb{Q} - \nabla_Q$ is a degree one endomorphism of $\gL$ which can only consist in multiplying sections by a degree one function $f$ on $\gM$. Thus,
 \[
 0 = L_\mathbb{Q} \omega = L_{\nabla_Q} \omega + f \omega = - d_\nabla i_{\mathbb{Q}} \omega + f \omega
 \]
 So that $f \omega = d_\nabla i_{\mathbb{Q} }\omega$. It follows that
 \[
 0 = d_\nabla (f \omega) = df \cdot \omega.
 \]
 Hence, by non-degeneracy, $df = 0$. Since $f$ is a positive degree function, one concludes that $f = 0$.
 \end{proof}
 
\begin{theorem} \label{TheoremLCS}
Every degree one abstract lcs $\N$-manifold $(\gM, \gL, \nabla, \omega)$ is of the kind $(T^\ast [1] M \otimes L, (T^\ast [1] M \otimes L) \times_M L, \nabla, d_\nabla \vartheta)$, up to isomorphisms (and a shift in the degree of $\gL$), where $\vartheta$ is the tautological degree one, $L$-valued $1$-form on $T^\ast[1] M \otimes L$ and $\nabla$ is a flat connection in the line bundle $L \rightarrow M$ (see Example \ref{Examp3}). Moreover, there is a one-to-one correspondence between degree one abstract lcs $\N Q$-manifolds and (non-graded) abstract locally conformal Poisson manifolds (see Appendix \ref{Ap1} for a definition).
\end{theorem}

\begin{proof}
The proof is a suitable adaptation of both the proofs of Theorem \ref{Gra} and Theorem \ref{Roy}. Let $(\gM, \gL, \nabla, \omega)$ be a degree one abstract lcs $\N$-manifold, and let $(D, \ell)$ be the Spencer data of $\omega$. In particular, $\gM = A[1]$ for some vector bundle $A \rightarrow M$, and $\gL = A[1] \times_M L $ for some line bundle $L \to M$ (up to a shift). Moreover $\nabla$ is induced in $\gL$ by a flat connection in $L$ which, abusing the notation, we denote by $\nabla$ again.

By non-degeneracy $\ell : \mathfrak{X}_{-1} (\gM) \rightarrow \Omega^1 (\gM , L)$ is an isomorphism $\Gamma (A) \simeq \Omega^1 (M , L)$, i.e.~$\gM = A[1] \simeq T^\ast [1] M \otimes L$. Moreover, since $\omega$ is $d_\nabla$-closed, diagram
\[
\xymatrix{\Gamma (A) \ar[r]^-D \ar[d]_-{\ell} & \Omega^2 (M , L) \\
                 \Omega^1 (M, L) \ar[ur]_-{-d_\nabla} &
} 
\]
commutes. This shows that diffeomorphism $\gM \simeq T^\ast[1] M \otimes L$ identifies $\omega$ with the canonical $L$-valued $2$-form on $T^\ast [1] M \otimes L$ (see Example \ref{Examp3}), thus proving the first part of the statement. In the following we identify $\gM$ and $T^\ast[1] M \otimes L$. For the second part of the statement, let $\mathbb{Q}$ be an homological derivation of $\gL = \gM \times_M L$. Derivation $\mathbb Q$ is equivalent to a Lie algebroid $(T^\ast M \otimes L, \rho, [\![-,-]\!])$ equipped with a representation $(L, \nabla^L)$ (beware not to confuse the Lie algebroid connection $\nabla^L$ and the standard connection $\nabla$). As above $(\gM, \gL, \nabla, \omega)$ is an abstract lcs $\N Q$-manifold with homological derivation $\mathbb{Q}$ iff  
(\ref{C2}), and (\ref{C3}) are satisfied for all
$X,Y \in \mathfrak{X}_{-1} (\gM) \simeq \Omega^1 (M, L)$. Similarly as in the proof of Theorem \ref{Roy}, one can even restrict to $X,Y$ in the form $d_\nabla \lambda,d_\nabla \mu$, with $\lambda,\mu \in \Gamma ( L)$. In this case one gets
\begin{align*}
B(X,Y) & = \ell ([[Q, d_\nabla \lambda], d_\nabla \mu]) + L_{[\mathbb Q, d_\nabla \lambda]} \ell (d_\nabla \mu) \\
                                        & = [\![d_\nabla \lambda,d_\nabla \mu]\!] + \nabla^L_{\rho (d_\nabla \lambda)}d_\nabla \mu\\
                                        & = [\![d_\nabla \lambda,d_\nabla \mu]\!] + d_\nabla \nabla^L_{\rho (d_\nabla \lambda)} \mu,
\end{align*}
and
\begin{align*}
C(X,Y) = i_{[Q,d_\nabla \lambda]} \ell (d_\nabla \mu) +  i_{[Q,d_\nabla \mu]} d_\nabla \lambda 
                                        = \nabla^L_{\rho (d_\nabla \lambda)} \mu + \nabla^L_{\rho (d_\nabla \mu)} \lambda,
\end{align*}
  where we used that $\ell (d_\nabla \lambda) = i_{d_\nabla \lambda} \omega = d_\nabla \lambda$, and $D(d_\nabla \lambda) = L_{d_\nabla \lambda} \omega = 0$ (see Example \ref{Examp1}). Concluding, $(\gM, \gL, \nabla, \omega)$ is an abstract lcs $\N Q$-manifold with homological derivation $\mathbb{Q}$ iff $\nabla^L_{\rho (d_\nabla \lambda)} \mu + \nabla^L_{\rho (d_\nabla \mu)} \lambda = 0$ and $[\![d_\nabla \lambda,d_\nabla \mu]\!] = - d_\nabla \nabla^L_{\rho (d_\nabla \lambda)} \mu = d_\nabla \nabla^L_{\rho (d_\nabla \mu)} \lambda$, i.e.~iff $(T^\ast M \otimes L, \rho, [\![-,-]\!])$ is the Lie algebroid associated to a locally conformal Poisson structure $(L, \nabla, P)$ on $M$, and $(L,\nabla^L)$ is its canonical representation.
\end{proof}

\section{Higher Degree Forms on Degree One $\N Q$-manifolds}\label{SecHigher}
\subsection{Vector valued forms on degree one $\N Q$-manifolds and Spencer operators}
In this section, we discuss general degree one compatible vector valued forms on degree one $\N Q$-manifolds. It turns out that they are equivalent to the recently introduced \emph{Spencer operators on Lie algebroids} \cite{css12}. Let $(A, [\![-,-]\!], \rho)$ be a Lie algebroid over a manifold $M$, $(E, \nabla^E)$ a representation of $A$, and let $k$ be a non-negative integer. An \emph{$E$-valued $k$-Spencer operator} \cite{css12} is a pair consisting of
\begin{itemize}
\item a (first order) differential operator $D: \Gamma(A) \rightarrow \Omega^k (M,E)$, and
\item a $C^\infty (M)$-linear map $\ell : \Gamma(A) \rightarrow \Omega^{k-1} (M,E)$,
\end{itemize}
such that
\[
D(fX)=f D(X) - df \wedge \ell (X)
\]
and, moreover,
\begin{align}
L_{\nabla^E_X} D(Y) - L_{\nabla^E_Y} D(X) & = D ([\![X,Y]\!]),  \label{IM3} \\
L_{\nabla^E_X} \ell (Y) + i_{\rho(Y)} D(X) & = \ell([\![X,Y]\!]), \nonumber\\
i_{\rho(X)} \ell (Y) + i_{\rho (Y)} \ell (X) & = 0, \nonumber
\end{align}
for all $X,Y \in \Gamma (A)$. There is a difference in signs between the above definition and the original one in \cite{css12}. The original definition is recovered by replacing $D \rightarrow -D$. We chose the sign convention which makes formulas simpler in the present graded context.

Spencer operators are the infinitesimal counterparts of multiplicative vector valued forms on Lie groupoids.  When the vector bundle is a trivial line bundle, they reduce to the IM forms of Bursztyn and Cabrera (see \cite{BCOb}, see also \cite{BCOa} for the $2$-form case).   Hence the result of this section is the expected generalization of the following (well) known facts:
\begin{itemize}
\item \emph{Jacobi manifolds can be understood either as infinitesimal counterparts of contact Lie groupoids \cite{cs13} or as degree one contact $\N Q$-manifolds \cite{m13}},
\item \emph{Poisson manifolds can be understood either as infinitesimal counterparts of symplectic Lie groupoids \cite{w87} or as degree one symplectic $\N Q$-manifolds \cite{r02}}.
\end{itemize}

\begin{theorem}\label{TheorSO} 
There is a one-to-one correspondence between
\begin{itemize}
\item degree one $\N$-manifolds equipped with a degree zero $\N Q$-vector bundle $\gE$ and a degree one compatible $\gE$-valued differential $k$-form, and 
\item Lie algebroids equipped with a representation $(E, \nabla^E)$ and an $E$-valued $k$-Spencer operator.
\end{itemize}
\end{theorem}

\begin{proof}
Let $\gM$ be a degree one $\N$-manifold, and let $(\gE, \mathbb{Q})$ be a degree zero $\N Q$-vector bundle over it. In particular $\gE = \gM \times_M E$ for a non-graded vector bundle $E \rightarrow M$. Let $(T^\ast M, \rho, [\![-,-]\!])$ and $(E, \nabla^E)$ be the Lie algebroid and the Lie algebroid representation corresponding to $\mathbb{Q}$. Finally, let $\omega$ be a degree one $E$-valued $k$-form on $\gM$. Then, $\omega$ is compatible with $\mathbb{Q}$  iff
(\ref{C1}), (\ref{C2}), and (\ref{C3}) are satisfied,
for all $X,Y \in \mathfrak{X}_{-1} (\gM) \simeq \Gamma (A)$. Denote by $(D, \ell)$ the Spencer data corresponding to $\omega$. Then\begin{align*}
A(X,Y) &= D([\![X,Y]\!]) - L_{\nabla^E_X} D(Y) + L_{\nabla^E_Y} D(X),\\
B(X,Y) &= \ell([\![X,Y]\!]) + i_{\rho (X)} D(Y) + L_{\nabla_Y^E} \ell(X),\\
C(X,Y) &= i_{\rho (X)} \ell (Y) + i_{\rho (Y)} \ell (X) .
\end{align*}
 
Concluding, $\omega$ is compatible with $\mathbb{Q}$ iff $(D,\ell)$ is an $E$-valued $k$-Spencer operator on the Lie algebroid $A$.
\end{proof}

\subsection{Degree one multisymplectic $\N Q$-manifolds}
We conclude this section specializing to degree one multisymplectic $\N Q$-manifolds.  Let $k$ be a positive integer. Recall that a $k$-plectic manifold (see, for instance, \cite{r12}, see also \cite{cid99} for more details on multisymplectic geometry) is a manifold $N$ equipped with a $k$-plectic structure, i.e.~a closed $(k+1)$-form $\omega$ which is non-degenerate in the sense that the vector bundle morphism $TN \rightarrow \wedge^k T^\ast N$, $X \mapsto i_X \omega$ is an embedding. As expected, degree one multisymplectic $\N Q$-manifolds are equivalent to Lie algebroids equipped with an \emph{IM multisymplectic structure}, also called \emph{higher Poisson structure} in \cite{bci13}. The latter are infinitesimal counterparts of multisymplectic groupoids. Specifically, an \emph{IM $k$-plectic structure} on a Lie algebroid $(A, [\![-,-]\!], \rho)$ (see \cite{bci13}) is a $C^\infty (M)$-linear map $\ell : A \rightarrow \Omega^k (M)$ such that 
\begin{align}
i_{\rho (X)} \ell (Y) + i_{\rho(Y)} \ell (X) & = 0, \label{IM1} \\ 
L_{\rho (X)} \ell (Y) - i_{\rho (Y)} d \ell(X) & = \ell ([\![ X,Y]\!]), \label{IM2}
\end{align}
for all $X,Y \in \Gamma (A)$, and, moreover,
\begin{itemize}
\item $\ker \ell := \{ a \in A : \ell (a)= 0 \} = 0$,
\item $(\operatorname{im} \ell )^\circ := \{\zeta \in  TM : i_\zeta \circ \ell = 0\} = 0$.
\end{itemize}

\begin{definition}
A degree $n$ \emph{$k$-plectic $\N$-manifold} is a degree $n$ $\N$-manifold $\gM$ equipped with a degree $n$ \emph{$k$-plectic structure}, i.e.~a closed $(k+1)$-form which is \emph{non-degenerate} in the sense that the degree $n$ vector bundle morphism $T \gM  \rightarrow S^k T^\ast [-1] \gM$, $X \mapsto i_X \omega$ is an embedding. A degree $n$ \emph{$k$-plectic $\N Q$-manifold} is an $\N Q$-manifold equipped with a compatible $k$-plectic structure.
\end{definition} 

\begin{example}
Let $M$ be an ordinary (non-graded) manifold. The degree $n$ $\N$-manifold $\gM = (\wedge^k T^\ast) [n] M$ comes equipped with the obvious tautological, degree $n$ $k$-form $\vartheta$. Consider the degree $n$ $(k+1)$-form $\omega = d\vartheta$. It is a degree $n$ $k$-plectic structure. Negatively graded vector fields on $\gM$ identify with $k$-forms on $M$ and it is easy to see, along similar lines as in Example \ref{Examp1}, that the Spencer data $(D,\ell)$ of $\omega$ identify with $(-)^n$ times the exterior differential $d : \Omega^k (M) \rightarrow \Omega^{k+1} (M)$ and the identity $\mr{id} : \Omega^k (M) \rightarrow \Omega^k (M)$ respectively.
\end{example}

\begin{theorem}\label{TheorMS}
Degree one $k$-plectic $\N Q$-manifolds are in one-two-one correspondence with Lie algebroids equipped with an IM $k$-plectic structure.
\end{theorem}

\begin{proof}
Let $\gM$ be a degree one $\N$-manifold, $\omega$ a degree one $(k+1)$-form on it and let $(D, \ell)$ be the corresponding Spencer data. In particular, $\gM = A[1]$ for some vector bundle $A \rightarrow M$. Moreover, $\omega$ is closed iff $i_X d\omega = 0$ for all negatively graded vector fields $X$ on $\gM$. Indeed, from $i_X d \omega = 0$ it also follows $L_X d \omega = 0$. In other words, $d \omega = 0$ iff diagram
\[
\xymatrix{\Gamma (A) \ar[r]^-D \ar[d]_-{\ell} & \Omega^{k+1} (M) \\
                 \Omega^k (M) \ar[ur]_-{-d} &
} 
\]
commutes. Conversely, a $C^\infty (M)$-linear map $\ell : \Gamma (A) \rightarrow \Omega^k (M)$ uniquely determines a closed degree one $(k+1)$-form on $\gM$ whose Spencer data are $(-d \circ \ell , \ell)$. Concluding, degree one $\N$-manifolds equipped with a closed $(k+1)$-form are equivalent to vector bundles $A \rightarrow M$ equipped with a linear map $\ell : \Gamma(A) \rightarrow \Omega^k (M)$.

Now, let $Q$ be an homological vector field on $\gM$ and let $(A, \rho, [\![-,-]\!])$ be the corresponding Lie algebroid. The $(k+1)$-form $\omega$ is compatible with $Q$ iff $(-d \circ \ell, \ell)$ is a $(k+1)$-Spencer operator, i.e.~$\ell$ fulfills (\ref{IM1}) and (\ref{IM2}) (Equation (\ref{IM3}) then follows from $D = -d \circ \ell$).

 Finally, we need to characterize non-degeneracy of the closed form $\omega$ in terms of $\ell$. Recall that $M$ can be understood as a submanifold in $\gM$ via the ``zero section'' of $\gM \rightarrow M$, and the vector bundle morphism $\Gamma : T \gM \rightarrow S^k T^\ast[-1] \gM$, $X \mapsto i_X \omega$, restricts to a vector bundle morphism $\Gamma |_M : T \gM|_M \rightarrow S^k T^\ast[-1] \gM |_M$. Now, there are canonical identification $T\gM|_M = TM \oplus A[1]$, and
\[
S^k T^\ast[-1] \gM |_M = \bigoplus_{i+j=k} \wedge^i T^\ast M \otimes S^j A^\ast [-1].
\]
It follows from $|\omega| = 1$ that $\Gamma |_M$ does actually take values in $ \wedge^{k-1} T^\ast M \otimes A^\ast [-1] \oplus \wedge^k T^\ast M$. More precisely, it identifies with the pair of vector bundle morphisms
\[
A[1] \longrightarrow \wedge^k T^\ast M, \quad X \longmapsto \ell(X).
\]
and
\[
TM \longrightarrow \wedge^{k-1} T^\ast M \otimes A^\ast [-1], \quad Z \longmapsto i_Z \circ \ell.
\]
Consequently, $\ker \ell$ and $(\operatorname{im} \ell)^\circ$ are trivial iff $\Gamma |_M$ is an embedding. It remains to show that $\omega$ is non-degenerate provided only $\Gamma |_M$ is an embedding. This is easily seen, for instance, in local coordinates: let $x^i$ be coordinates in $M$ and $z^a$ be (degree one) fiber coordinates in $A[1] \rightarrow M$. Locally,
\[
\omega = \omega_{a| i_1 \cdots i_k} dz^a dx^{i_1} \cdots dx^{i_k} + \omega^\prime_{a | i_1 \cdots i_{k+1}} z^a dx^{i_1} \cdots dx^{i_{k+1}}.
\]
In the basis $\{ \; \partial / \partial z^a \; | \; \partial / \partial x^i \; \}$ of $\frak{X} (\gM)$ and $\{\; dx^{i_1} \cdots dx^{i_k} \; | \; dz^a dx^{i_1} \cdots dx^{i_{k-1}} \; | \; \cdots \; \}$ of $\Omega^k (\gM)$, vector bundle morphism $\Gamma$ is represented by matrix
\[
\left( \begin{array}[c]{c}
\xymatrix@R=45pt{ \  \\ \ 
}
\end{array}
 \right. \hspace{-30pt}
\begin{array}[c]{c}
 \xymatrix@C=5pt@R=5pt{  &  &  \ar@{.}[dddd] &  &  \ar@{.}[dddd]  &  & \\
                   & \omega_{a | i_1 \cdots i_k} & & 0 &  & \cdots &\\
               \ar@{.}[rrrrrr]   &  &  &  &  & & \\
                  & \ast & & k \omega_{a | we i_1 \cdots i_k} & &  \cdots &\\
                  & & & & & &}
                  \end{array}
\left. \hspace{-25pt} \begin{array}[c]{c}
\xymatrix@R=45pt{ \  \\ \ 
}
\end{array}
 \right),
\]
and $\Gamma |_M$ is represented by the same matrix with the lower-left block set to zero. This concludes the proof. 
\end{proof}

 \appendix
 
\section{Locally Conformal Symplectic Manifolds Revisited}\label{Ap1}
We refer to \cite{v85} for details about standard locally conformal symplectic (lcs) structures. Here, we present a slightly more intrinsic approach to them \cite{v14} (see also \cite[Section 3]{vit15}). Let $M$ be a smooth manifold.
\begin{definition}\label{def1}
An \emph{abstract lcs structure} on $M$ is a triple $(L, \nabla, \omega)$, where $L \rightarrow M$ is a line bundle, $\nabla$ is a flat connection in $L$, and $\omega$ is a non-degenerate $L$-valued $2$-form on $M$ such that $d_\nabla \omega = 0$, where $d_\nabla : \Omega (M,L) \rightarrow \Omega (M,L)$ is the de Rham differential determined by $\nabla$. A manifold equipped with an abstract lcs structure is an \emph{abstract lcs manifold}.
\end{definition}
\begin{example}
Let $(L, \nabla, \omega)$ be an abstract lcs structure on $M$. If $L = M \times \mathbb{R}$ is the trivial line bundle, then $\nabla$ is the same as a closed $1$-form on $M$, specifically, the \emph{connection $1$-form $\phi := - d_\nabla 1 \in \Omega^1 (M)$}. Moreover, $\omega$ is a standard (non-degenerate) $2$-form on $M$ and it is easy to see that $(\phi, \omega)$ is a standard lcs structure, i.e.~$d\omega = \phi \wedge \omega$. In particular, if $\phi = 0$, then $\omega$ is a symplectic structure.
\end{example}
The word ``abstract'' in Definition \ref{def1} refers to the fact that $\omega$ takes values in an ``abstract'' line-bundle $L$, as opposed to the concrete, trivial line bundle $M \times \mathbb{R}$. Similarly, one can define \emph{``abstract'' locally conformal Poisson manifolds} (see below) and, more generally, \emph{``abstract'' Jacobi manifolds}. An \emph{abstract Jacobi structure} (called a Jacobi bundle in \cite{Marle1991}) on a manifold $M$ is a line bundle $L$ equipped with a Lie bracket $\{-,-\}$ on $\Gamma (L)$ which is a first order differential operator in each entry (see, e.g., \cite{cs13} for details). Abstract Jacobi manifolds where introduced by Kirillov \cite{Kir} under the name \emph{local Lie algebras with one dimensional fibers}. An abstract lcs structure $(L, \nabla, \omega)$ on $M$ determines an abstract Jacobi structure $(L, \{-,-\})$ as follows. First of all, by non degeneracy, $\omega$ establishes an isomorphism $TM \rightarrow T^\ast M \otimes L$, $X \mapsto i_X \omega$. Denote by $\sharp: T^\ast M \otimes L \rightarrow TM$ the inverse isomorphism and, for $\lambda \in \Gamma (L)$, put $X_\lambda := \sharp (d_\nabla \lambda) \in \mathfrak{X} (M)$. Finally, put
\[
\{ \lambda, \mu \} := \omega (X_\lambda, X_\mu) = \nabla_{X_\lambda} \mu .
\] 
$\lambda, \mu \in \Gamma (L)$. Clearly, $\{-,-\}$ is a first order differential operator in each entry. Moreover, the Jacobi identity is equivalent to $d_\nabla \omega = 0$. Thus, $(L, \{-,-\})$ is an abstract Jacobi structure on $M$. Notice that there exists a unique linear morphism $P : \wedge^2 (T^\ast M \otimes L) \rightarrow L$ such that
\[
\{ \lambda , \mu \} = P (d_\nabla \lambda, d_\nabla \mu ), \quad \text{for all } \lambda, \mu \in \Gamma (L) .
\]
\begin{example}
Let $L = M \times \mathbb{R}$ so that $(L, \nabla, \omega)$ is the same as a standard lcs structure $(\phi, \omega)$. Then, for $f,g \in C^\infty (M) = \Gamma (L)$, $X_f$ is implicitly defined by
\[
i_{X_f} \omega = df - f \phi ,
\]
and 
\[
\{ f, g \} := \omega (X_f, X_g) = X_f (g) - g \phi (X_f).
\]
In particular, if $\phi = 0$, then $P$ is the Poisson bivector determined by the symplectic structure $\omega$.
\end{example}

More generally, Let $M$ be a smooth manifold, $(L, \nabla)$ a line bundle over $M$ equipped with a flat connection, and let $P : \wedge^2 (T^\ast M \otimes L) \rightarrow L$ be a linear morphism. One can then define a bracket $\{-, -\}_P$ in $\Gamma (L)$ by putting
\[
\{ \lambda , \mu \}_P = P (d_\nabla \lambda, d_\nabla \mu ),
\] 
 $\lambda, \mu \in \Gamma (L)$.
 \begin{definition}
 An \emph{abstract locally conformal Poisson structure} on $M$ is a triple $(L, \nabla, P)$, where $L \rightarrow M$ is a line bundle, $\nabla$ is a flat connection in $L$, and $P$ is a linear morphism $P : \wedge^2 (T^\ast M \otimes L) \rightarrow L$ such that $\{-,-\}_P$ is a Lie bracket. A manifold equipped with an abstract locally conformal Poisson structure is an \emph{abstract locally conformal Poisson manifold}.
 \end{definition}
 
 Thus, abstract lcs manifolds are abstract locally conformal Poisson manifolds (much as standard symplectic manifolds are standard Poisson manifolds), but the latter are more general. 
 
 \begin{example}
 Let $(L, \nabla, P)$ be an abstract locally conformal Poisson structure on $M$. If $L = M \times \mathbb{R}$ is the trivial line bundle, and $\phi := - d_\nabla 1 \in \Omega^1 (M)$ is the connection $1$-form, then $P$ is a standard bivector on $M$ and a lengthy but straightforward computation shows that $(\phi, P)$ is a \emph{locally conformal Poisson structure} in the sense of \cite{v07}, i.e.~$[P,P]_{ns} = i_{\phi} P \wedge P$ (where $[-,-]_{ns}$ is the Nijenhuis-Schouten bracket of multivectors). In particular, if $\phi = 0$, then $P$ is a Poisson structure.
 \end{example}
 
 Finally, notice also that abstract locally conformal Poisson manifolds are abstract Jacobi manifolds (of a special kind).

 \section{Lie Algebroids and Their Representations}\label{ApLA}

Recall that a \emph{Lie algebroid} over a manifold $M$ is a vector bundle
$A\rightarrow M$ equipped with 1) a $C^{\infty}(M)$-linear map
$\rho :\Gamma(A)\rightarrow\mathfrak{X}(M)$ called the \emph{anchor},
and 2) a Lie bracket $[\![-,-]\!]$ on $\Gamma(A)$ such that
\[
[\![ X,fY]\!]=\rho(X)(f)Y+f[\![X,Y]\!],\quad
X,Y\in\Gamma(A),\quad f\in C^{\infty}(M).
\]
\begin{example}
The tangent bundle $TM$ is a Lie algebroid with Lie bracket given by the commutator of vector fields and anchor given by the identity.
\end{example}
Let $A\rightarrow M$ be a Lie algebroid. A
\emph{representation of} $A$ is a vector bundle $E \rightarrow M$ equipped
with a \emph{flat }$A$\emph{-connection} $\nabla^E$, i.e.~a $C^{\infty}%
(M)$-linear map $\nabla^E:\Gamma(A)\rightarrow \Gamma (D E)$, denoted
$X \mapsto\nabla^E_{X}$, such that the symbol of the derivation $\nabla^E_X$ is $\rho(X)$, and $[\nabla^E_{X},\nabla^E_{Y}]=\nabla^E
_{[\![X,Y]\!]}$, for all $X,Y\in\Gamma(A)$. Let $(E,\nabla^E)$ be a
representation of $A$. The graded vector space $\Gamma (\wedge^\bullet A^\ast \otimes E)$ of alternating, $C^{\infty
}(M)$-multilinear, $\Gamma(E)$-valued forms on $\Gamma(A)$ is naturally
equipped with an homological operator $d_{E}$ given by the following
\emph{Chevalley-Eilenberg formula}:
\begin{align*}
&  (d_{E}\varphi)(X_{1},\ldots,X_{k+1})\\
&  :=\sum_{i}(-)^{i+1}\nabla^E_{X_{i}}(\varphi(\ldots,\widehat{X_{i}%
},\ldots))+\sum_{i<j}(-)^{i+j}\varphi([\![X_{i},X_{j}]\!],\ldots
,\widehat{X_{i}},\ldots,\widehat{X_{j}},\ldots),
\end{align*}
where $\varphi\in\Gamma (\wedge^k A^\ast \otimes E)$ is an alternating
form with $k$-entries, $X_{1},\ldots,X_{k+1}\in\Gamma(A)$, and a hat
$\widehat{\left(  -\right)  }$ denotes omission.

\begin{example}
Let $\nabla$ be a standard flat connection in a vector bundle $E$. Then $(E , \nabla)$ is a representation of the Lie algebroid $TM$ and the de Rham operator $d_\nabla$ of $\nabla$ is its associated homological operator.
\end{example}

\begin{example}
Let $A \rightarrow M$ be a Lie algebroid. Clearly $(M \times \mathbb{R}, \rho)$ is a canonical representation of $A$. In particular, $\Gamma (\wedge^\ast A)$ is equipped with an homological operator (in fact a derivation) which we denote by $d_A$.
\end{example}

\begin{example}
Let $(L, \{-,-\})$ be an abstract Jacobi structure on a manifold $M$. There is a unique Lie algebroid $(J^1 L , \rho, [\![-, -]\!])$ such that $[\![ j^1 \lambda, j^1 \mu ]\!] = j^1 \{ \lambda, \mu \}$, and $\rho (j^1 \lambda)$ is the symbol of the first order differential operator (in fact a derivation) $\{ \lambda , -\}$, where $\lambda, \mu \in \Gamma (L)$. Moreover, there is a unique representation $(L, \nabla^L )$ of $J^1 L$ such that $\nabla^L_{j^1 \lambda} \mu = \{ \lambda, \mu \}$. In particular, 
\begin{equation} \label{Eq15}
[\![ j^1 \lambda, j^1 \mu ]\!] = j^1 \left( \nabla^L_{j^1 \lambda} \mu \right).
\end{equation}
Conversely, let $(J^1 L, \rho, [\![ -, -]\! ] )$ be a Lie algebroid equipped with a representation $(L, \nabla^L)$ such that (\ref{Eq15}) holds. For $\lambda, \mu \in \Gamma (L)$ put $\{ \lambda , \mu \} := \nabla^L_{j^1 \lambda} \mu$. Then $(L, \{-,-\})$ is an abstract Jacobi structure on $M$. This shows that abstract Jacobi structures $(L, \{-,-\})$ are equivalent to Lie algebroids $(J^1 L, \rho, [\![ -, -]\! ] )$ equipped with a representation $(L, \nabla^L)$ such that (\ref{Eq15}) holds.
\end{example}

\begin{example}
Let $\{-,-\}$ be a Poisson structure on a manifold $M$. There is a unique Lie algebroid $(T^\ast M , \rho, [\![-, -]\!])$ such that $[\![ df, dg ]\!] = d \{ f, g \}$, and $\rho (df)$ is the Hamiltonian vector field of $f$, where $f,g \in C^\infty (M)$. In particular, 
\begin{equation} \label{Eq16}
[\![ df, dg ]\!] = d \left( \rho (df) (g) \right) \quad \text{and} \quad \rho (df) (g) + \rho (dg) (f) = 0.
\end{equation}
Conversely, let $(T^\ast M, \rho, [\![ -, -]\! ] )$ be a Lie algebroid such that (\ref{Eq16}) hold. For $f,g \in C^\infty (M)$ put $\{ f , g \} :=  \rho (df) (g) $. Then $\{-,-\}$ is a Poisson structure on $M$. This shows that Poisson structures are equivalent to Lie algebroids $(T^\ast M, \rho, [\![ -, -]\! ] )$ such that (\ref{Eq16}) hold.
\end{example}

\begin{example}
Let $(L, \nabla, \omega)$ be an abstract locally conformal Poisson structure on a manifold $M$ (see the previous appendix). There is a unique Lie algebroid $(T^\ast M \otimes L , \rho, [\![-, -]\!])$ such that $[\![ d_\nabla \lambda, d_\nabla \mu ]\!] = d_\nabla \{ \lambda, \mu \}$, and $ \rho (d_\nabla \lambda)$ is the symbol of the first order differential operator $\{ \lambda , -\}$, where $\lambda, \mu \in \Gamma (L)$. Moreover, there is a unique representation $(L, \nabla^L )$ of $T^\ast M \otimes L$ such that $\nabla^L_{d_\nabla \lambda} \mu = \{ \lambda, \mu \}$.
In particular, 
\begin{equation} \label{Eq17}
[\![ d_\nabla \lambda, d_\nabla \mu ]\!] = d_\nabla \left( \nabla^L_{ d_\nabla \lambda} \mu \right) \quad \text{and} \quad \nabla^L_{d_\nabla \lambda} \mu + \nabla^L_{d_\nabla \mu} \lambda = 0.
\end{equation}
Conversely, let $(T^\ast M \otimes L, \rho, [\![ -, -]\! ] )$ be a Lie algebroid equipped with a representation $(L, \nabla^L)$ such that (\ref{Eq17}) hold. For $\lambda, \mu \in \Gamma(L)$ put $\{ \lambda , \mu \} :=  \nabla^L_{ d_\nabla \lambda} \mu $. Then $(L, \nabla, \{-,-\})$ is a locally conformal Poisson structure on $M$. This shows that locally conformal Poisson structures are equivalent to Lie algebroids $(T^\ast M \otimes L, \rho, [\![ -, -]\! ] )$ such that (\ref{Eq17}) hold.
\end{example}

\subsection*{Acknowledgement} I thank the anonymous referee for carefully reading the first manuscript and for her/his suggestions to improve the readability of the paper.

 \bigskip

\end{document}